\documentclass{amsart}

\usepackage{amssymb}
\usepackage{amsmath}
\usepackage{latexsym}
\usepackage{mathrsfs}
\usepackage[small,nohug,heads=vee]{diagrams}
\usepackage{amsthm}
\usepackage{verbatim}
\usepackage{graphicx}
\usepackage{epstopdf}
\usepackage{epsfig}
\usepackage{color}
\usepackage{breqn}
\usepackage{subfigure}
\usepackage{float}
\usepackage{extarrows}
\usepackage{color}
\usepackage{bm}
\usepackage[colorlinks,linkcolor=blue,anchorcolor=green,citecolor=red]{hyperref}
\usepackage{times}
\usepackage{fancyhdr}
\usepackage{extarrows}
\usepackage{amscd}

\newtheorem{theorem}{Theorem}
\newtheorem{remark}{Remark}

\newtheorem{lemma}{Lemma}
\newtheorem{problem}{Problem}

\newtheorem{algorithm}{Algorithm}

\makeatletter
\@addtoreset{equation}{section}
\makeatother

\usepackage{geometry}
\begin{document}
\title{Structure-preserving Finite Element Methods
for Stationary  MHD Models}

\author{Kaibo Hu}
\address{Beijing International Center for Mathematical Research, Peking University, Beijing 100871, P. R. China.}
\email{kaibo@pku.edu.cn}
\thanks{}

\author{Jinchao Xu}
\address{Center for Computational Mathematics and
  Applications and Department of Mathematics, The Pennsylvania State
  University, University Park, PA 16802, USA.}
\email{xu@math.psu.edu}
 \thanks{This material
    is based upon work supported in part by the US Department of
    Energy Office of Science, Office of Advanced Scientific Computing
    Research, Applied Mathematics program under Award Number
    DE-SC-0014400 and by Beijing International Center for Mathematical
    Research of Peking University, China.}


\date{}

\begin{abstract}
  In this paper, we develop a class of mixed finite element scheme for
  stationary magnetohydrodynamics (MHD) models, using magnetic field
  $\bm B$ and current density $\bm j$ as the discretization variables.
  We show that the Gauss's law for the magnetic field, namely
  $\nabla\cdot\bm{B}=0$, and the energy law for the entire system are
  exactly preserved in the finite element schemes.  Based on some new
  basic estimates for $H^{h}(\mathrm{div})$, we show that the new
  finite element scheme is well-posed.  Furthermore, we show the
  existence of solutions to the nonlinear problems and the convergence
  of Picard iterations and finite element methods under some conditions.
\end{abstract}

\maketitle

\section{Introduction}

In this paper, we develop structure-preserving finite element
discretization for the following stationary incompressible
magnetohydrodynamics (MHD) system:
\begin{dgroup}[compact]
\label{mhd}
\begin{dmath}
( \boldsymbol{u} \cdot \nabla) \boldsymbol{u}
-R_{e}^{-1} \Delta \boldsymbol{u}
- S \boldsymbol{j} \times \boldsymbol{B}
+ \nabla p
= \boldsymbol{f} , 
\label{mhd1} 
\end{dmath} 
\begin{dmath}
\boldsymbol{j}
- R_{m}^{-1} \nabla \times \boldsymbol{B}
= \boldsymbol{0} ,
\label{mhd2} 
\end{dmath}
\begin{dmath}
\nabla \times \boldsymbol{E} = \boldsymbol{0} ,
\label{mhd3}
\end{dmath}
\begin{dmath}
\nabla \cdot \boldsymbol{B} = 0,
\label{mhd4} 
\end{dmath}
\begin{dmath}
\nabla \cdot \boldsymbol{u} = 0,
\label{mhd5} 
\end{dmath}
\end{dgroup}
where the Ohm's law holds:
\begin{equation}
\label{ohm1}
\bm{j}= \bm{E}+\bm{u}\times \bm{B}.
\end{equation}
Here $\bm{u}$ is the velocity of conducting fluids, $p$ is the
pressure, $\bm{B}$ is the magnetic field, $\bm{E}$ is the electric
field and $\bm{j}$ is the volume current density.  Dimensionless
parameters $R_{e}$, $R_{m}$ and $S$ are the Reynolds number of fluids,
magnetic field and the coupling number respectively.

In the study of magnetohydrodynamics (MHD) system, it is well-known
that the Gauss's law for the magnetic field, namely $\nabla\cdot
\bm{B}=0$, is an important condition in numerical computation of MHD
system \cite{Brackbill.J;Barnes.D.1980a,
  Dai.W;Woodward.P.1998b}. Nonzero divergence of $\bm{B}$ will
introduce a parallel force, which breaks the energy law.  In our
previous work Hu, Ma and Xu \cite{hu2014stable}, we proposed a
class of structure-preserving and energy-stable finite element
discretizations that exactly preserve the magnetic Gauss's law on the
discrete level for the time dependent MHD systems.  The goal of this
paper is to extend such discretizations to stationary cases.

Such a discretization is however not straightforward as the
time-dependent and the stationary systems have different structures.
In the time-dependent problem, the Faraday's law reads:
$$
\frac{\partial \bm{B}}{\partial t}+\nabla\times \bm{E}=\bm{0}.
$$
In \cite{hu2014stable}, we chose to keep the electric field
$\bm{E}$ and use the $H(\mathrm{curl})$-conforming finite element space
for $\bm{E}$ and $H(\mathrm{div})$-conforming finite element space for $\bm{B}$
to discretize the above Faraday's law as follows:
$$
\frac{\bm{B}^n-\bm{B}^{n-1}}{\Delta t}+\nabla\times \bm{E}^n=\bm{0}.
$$
This implies that $\nabla\cdot \bm{B}^n=0$ holds for all $n\ge 1$ as
long as it holds for $n=0$. 

In  the stationary case,  the Faraday's law reads:
$$
\nabla\times\bm{E}=\bm{0}.
$$
In this case, we can not directly apply the technique used in
\cite{hu2014stable} for the evolutionary case to preserve the
Gauss's law $\nabla\cdot \bm{B}=0$ exactly on the discrete level.
Instead we treat the Gauss's law as an independent equation in the
whole MHD system and we then introduce a Lagrange multiplier to
appropriately enforce this law on both the continuous and the discrete
level.

The idea of the use of Lagrange multiplier itself is not new (see
Sch\"{o}tzau \cite{Schotzau.D.2004a} and the reference therein) and
the novelty of our approach here lies in how this technique is used in
combination with the techniques developed in \cite{hu2014stable}.
In Sch\"{o}tzau \cite{Schotzau.D.2004a}, a magnetic multiplier $r\in
H^{1}(\Omega)/\mathbb{R}$ is used to impose the Gauss's law in the
following way:
$$
\int_{\Omega} \bm{B}\cdot \nabla s=0,\quad \forall s\in H^{1}(\Omega)/\mathbb{R}
$$
which does not guarantee that the Gauss's law holds strongly (namely
$\nabla\cdot\bm{B}_h=0$ point-wise in the domain) in the corresponding
discrete case.  The main difference in our approach is that the
Gauss's law will indeed be preserved on the discrete level strongly by
using appropriate finite element discretization of $\bm{B}$ so that
$\bm{B}_{h}$ is $H(\mathrm{div})$-conforming. On the other hand, the charge conservation $\nabla\cdot\bm{j}=0$ is preserved in a weak sense.  The finite
element de Rham sequence as studied in
\cite{Arnold.D;Falk.R;Winther.R.2006a, Hiptmair.R.2002a, Bossavit.A.1998a}
plays an important role in the construction and analysis in our paper.

MHD equations admit many different variational formulations which lead
to different mathematical properties and numerical efficiency on the
discrete level.  In most existing literature, variables $\bm{E}$ and
$\bm{j}$ are eliminated to reduce the size of the corresponding
discretized problems.  In \cite{hu2014stable}, we demonstrated
that it is advantageous to keep $\bm{E}$ and use it as an independent (or
intermediate) discretization variable in appropriate finite element
space.  Indeed, this approach may lead to larger discretized systems,
but these systems have better mathematical structures and may be
solved, as illustrated in \cite{ma2016robust}, more efficiently
than the corresponding smaller systems derived from traditional
schemes by eliminating both $\bm{E}$ and $\bm{j}$.

In this paper, we continue and extend this study for the stationary
problem. Instead of retaining $\bm{E}$ explicitly as a variable, we choose
$\bm{B}$ and $\bm{j}$ as electromagnetic variables motivated by
the energy law.


For simplicity of exposition, we use the following homogeneous
Dirichlet boundary conditions
\begin{align*}
\bm{u}&=\bm{0},\\
\bm{B}\cdot \bm{n}&=0,\\
\bm{j}\times \bm{n}&=\bm{0}.
\end{align*}
According to the Ohm's law that $\bm{j}=\bm{E}+\bm{u}\times \bm{B}$,
the above boundary conditions are obviously equivalent to
\begin{align*}
\bm{u}&=\bm{0},\\
\bm{B}\cdot \bm{n}&=0,\\
\bm{E}\times \bm{n}&=\bm{0}.
\end{align*}

The extension to non-homogeneous  boundary conditions  is straightforward and standard and
the relevant details will not be given in this paper.

The rest of the paper is organized as follows. In \S
\ref{sec:notation}, we present the notation and basic finite element
spaces used in the discussion. \S \ref{sec:estimates} demonstrates basic estimates
for $H^{h}(\mathrm{div}0)$ functions, including regularity and the discrete Poincar\'{e}'s inequality.
 In \S
\ref{sec:Bj}, a new formulation based on $\bm{B}$ and $\bm{j}$ is
studied. We prove the well-posedness based on an equivalent reduced system. In \S \ref{sec:convergence}, we prove the analysis of the proposed algorithms based on the key technical results established in \S \ref{sec:estimates}. This includes the convergence of Picard iterations and  the finite element discretizations.  Concluding remarks are given in \S \ref{sec:conclusion}.

\section{Notation and basic finite element spaces}\label{sec:notation}
In this section, we introduce some basic Sobolev spaces and their
corresponding finite element discretizations that will be used in the
rest of the paper.

{ We assume that $\Omega$ is a bounded Lipschitz
  polyhedron. For the ease of exposition, we further assume that
  $\Omega$ is contractable, i.e. there is no nontrivial harmonic
  form. 
  For general domains (non-simply-connected domain, non-connected
  boundary), we can solve the problem in the orthogonal complement of
  (discrete) harmonic forms, as in Arnold, Falk and Winther
  \cite{Arnold.D;Falk.R;Winther.R.2006a} for the Hodge
  Laplacian. Therefore such an assumption on the domain is to make the
  presentation more clear, and the methodology is also valid for
  general topology.}

 Using the standard notation for inner product and norm of the
{$L^{2}$} space
$$
(u,v):=\int_{\Omega}u\cdot v \mathrm{d}x,\quad
\|u\|:=\left(\int_{\Omega} \lvert u\rvert^2 \mathrm{d}x\right)^{1/2},  
$$
we define the following $H(D,\Omega)$ space with a given linear operator {$D$}:
$$
H(D,\Omega):=\{v\in L^2(\Omega), Dv\in L^2(\Omega)\},  
$$
and
$$
H_0(D,\Omega):=\{v\in H(D, \Omega), t_{D}v=0 \mbox{ on } \partial\Omega\},
$$
where $t_{D}$ is the trace operator:
$$
t_{D}v:=
\left\{
  \begin{array}{cc}
    v, & D=\mathrm{grad},\\
    v\times n, & D=\mathrm{curl},\\
    v\cdot n, & D=\mathrm{div}.
  \end{array}
\right.
$$
Here $H(\mathrm{grad}, \Omega)$ is a scalar function space, while $H(\mathrm{curl}, \Omega)$ and $H(\mathrm{div}, \Omega)$ are for vector valued functions. We
often use the following notation:
$$
L^2_0(\Omega):=\left\{v\in L^2(\Omega):  \int_\Omega v=0 \right\}.
$$
When {$D=\mathrm{grad}$}, we often use the notation:
$$
H^1(\Omega):=H(\mathrm{grad}, \Omega), \quad
H^1_0(\Omega):=H_0(\mathrm{grad}, \Omega).
$$
For clarity, the corresponding norms in $H(D, \Omega)$ are denoted by
$$
\|\bm{u}\|_{1}^{2}=\|\bm{u}\|^{2}+\|\nabla \bm{u}\|^{2},
$$
$$
\|\bm{F}\|_{\mathrm{curl}}^{2}:=\|\bm{F}\|^{2}+\|\nabla\times \bm{F}\|^{2},
$$
$$
\|\bm{C}\|_{\mathrm{div}}^{2}:= \|\bm{C}\|^{2}+\|\nabla\cdot \bm{C}\|^{2}.
$$

We will also use the space {$L^{p}$} with norm $\|\cdot\|_{0,p}$ given
by $\|{v}\|_{0,p}^p= \int_\Omega|v|^p$.  For a general Banach space
$\bm{Y}$ with a norm $\|\cdot\|_{\bm{Y}}$, the dual space
$\bm{Y}^{\ast}$ is equipped with the dual norm defined as
$$
\|\bm{h}\|_{\bm{Y}^{\ast}}:=\sup_{0 \neq \bm{y}\in \bm{Y}}\frac{\langle \bm{h}, \bm{y} \rangle}{\|\bm{y}\|_{\bm{Y}}}.
$$
For the special case that $\bm{Y}=H_0^1(\Omega)$,
$\bm{Y}^\ast=H^{-1}(\Omega)$ and the corresponding norm is denoted by
$\|\cdot\|_{-1}$, which is defined as
$$
\|\bm{f}\|_{-1}:=\sup_{0 \neq \bm{v}\in {H}_{0}^{1}(\Omega)^{3}}\frac{\langle \bm{f}, \bm{v} \rangle}{\|\nabla\bm{v}\|}.
$$

We will use $C_{1}$ to denote the constant in the following inequality, which is a consequence of Sobolev imbedding theorem and Poincar\'{e}'s inequality:
\begin{align}\label{def:C1}
\|{u}\|_{0,6}\leq C_{1}\|\nabla{u}\|,\quad \forall {u}\in H^{1}_{0}(\Omega).
\end{align}

Since the fluid convection frequently appears in the following
discussions, we introduce the trilinear form
$$
L(\bm{w}; \bm{u}, \bm{v}):=\frac{1}{2}[((\bm{w}\cdot \nabla)\bm{u},
\bm{v})- ((\bm{w}\cdot\nabla) \bm{v}, \bm{u}) ].
$$
When $\bm{w}$ is a known function, $L(\bm{w}; \bm{u}, \bm{v})$ is a
bilinear form of $\bm{u}$ and $\bm{v}$. This will occur in the Picard
iteration, where $\bm{w}$ is the velocity of the last iteration step.

Let $\mathcal{T}_{h}$ be a triangulation of $\Omega$, and we assume
that the mesh is regular and quasi-uniform, so that the inverse
estimates hold \cite{brenner2008mathematical}.  The finite element de
Rham sequence is an abstract framework to unify the above spaces and
their discretizations, see e.g. Arnold, Falk, Winther
\cite{Arnold.D;Falk.R;Winther.R.2006a,Arnold.D;Falk.R;Winther.R.2010a},
Hiptmair \cite{Hiptmair.R.2002a}, Bossavit \cite{Bossavit.A.1998a} for
more detailed discussions. Figure~\ref{exact-sequence} shows the
commuting diagrams we will use. Current density $\bm{j}$, magnetic
field $\bm{B}$ and the multiplier $r$ will be discretized in the last
three spaces respectively. Figure~\ref{deRham} shows the finite
elements of the lowest order.
\begin{figure}[ht!]
\begin{equation*}
\begin{CD}
H_0(\mathrm{grad})   @> {\mathrm{grad}} >> H_0(\mathrm{curl}) 
@>{\mathrm{curl}} >> H_0(\mathrm{div})  @> {\mathrm{div}} >> L_0^2  \\    
 @VV\Pi_h^{\mathrm{grad}} V @VV\Pi_h^{\mathrm{curl}} V @VV\Pi_h^{\mathrm{div}} V
@VV\Pi_h^0 V\\  
H^h_0(\mathrm{grad})  @>{\mathrm{grad}} >> H^h_0(\mathrm{curl})  @>{\mathrm{curl}} >> H^h_0(\mathrm{div}) @> {\mathrm{div}} >> L^{2,h}_0
\end{CD}
\end{equation*}
\caption{Continuous and discrete de Rham sequence}
\label{exact-sequence}
\end{figure}
\begin{figure}[ht!]
\begin{center}
\includegraphics[width=.8in]{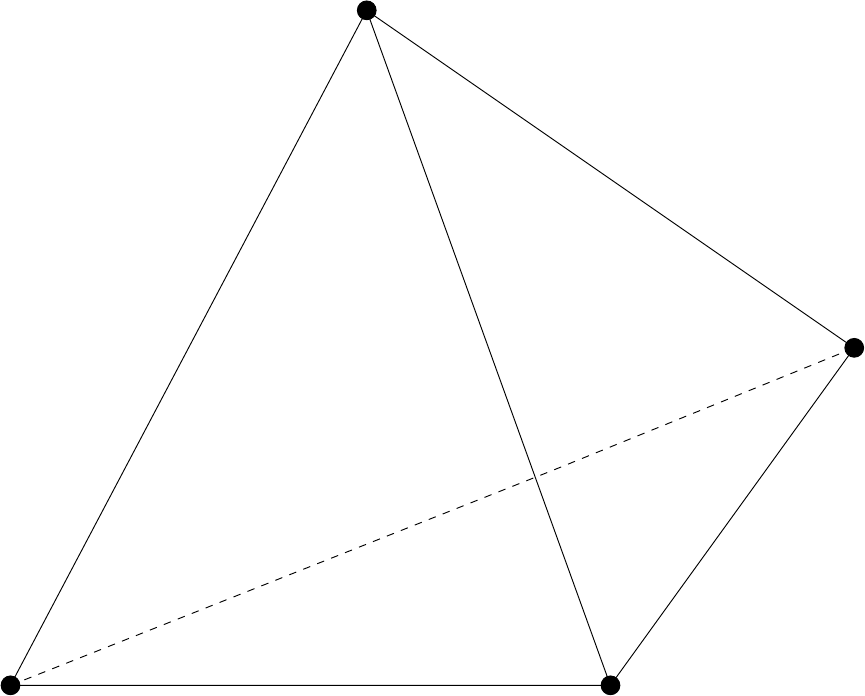}
\includegraphics[width=.8in]{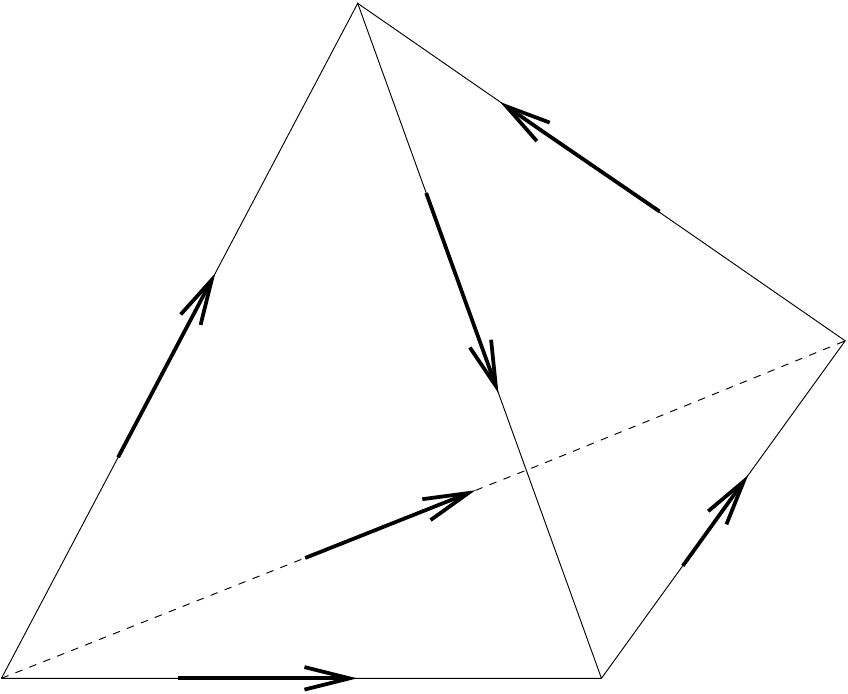}
\includegraphics[width=.8in]{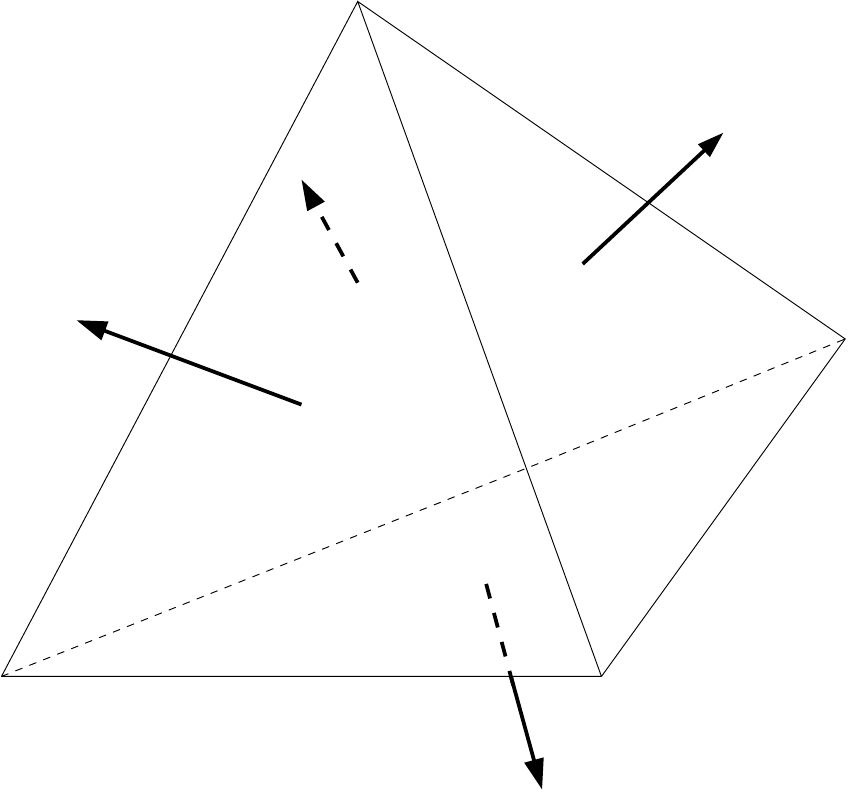}
\includegraphics[width=.8in]{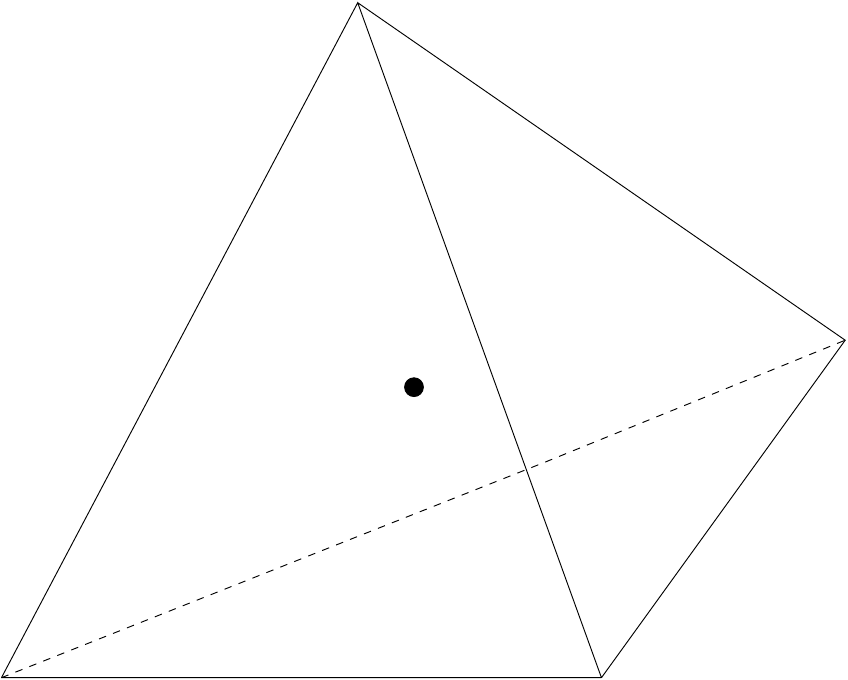}
\caption{DOF of finite element de Rham sequence of lowest order}
\label{deRham}
\end{center}
\end{figure}

As we shall see, $H(\mathrm{div})$ functions with vanishing divergence will play an important role in the study. So we define on the continuous level
$$
H_{0}(\mathrm{div}0, \Omega):=\{\bm{C}\in H_{0}(\mathrm{div}, \Omega): \nabla\cdot\bm{C}=0  \},
$$
and the finite element subspace
$$
H_{0}^{h}(\mathrm{div}0, \Omega):=\{\bm{C}_{h}\in H_{0}^{h}(\mathrm{div}, \Omega): \nabla\cdot\bm{C}_{h}=0  \}.
$$

 We use $\bm{V}_{h}$ to denote the finite element
subspace of velocity $\bm{u}_{h}$, and $Q_{h}$ for pressure
$p_{h}$. There are many existing stable pairs for $\bm{V}_{h}$ and
$Q_{h}$, for example, Taylor-Hood elements
\cite{Girault.V;Raviart.P.1986a, Boffi.D;Brezzi.F;Fortin.M.2013a}.
Spaces $H^{h}_{0}(\mathrm{div}, \Omega)$ and $ L_{0,h}^{2}(\Omega)$ are finite
element spaces from the discrete de Rham
sequence. For these spaces we use their explicit names for clarity, and use the notation 
$\bm{V}_{h}$ and $\bm{Q}_{h}$ for the fluid part to indicate that they are usually different from 
$H_{0,h}^{1}(\Omega)^{3}$ and $L^{2}_{0,h}(\Omega)$ in the de Rham sequence.

There is a unified theory for the discrete de Rham sequence  of arbitrary
order \cite{Boffi.D;Brezzi.F;Fortin.M.2013a,
  Arnold.D;Falk.R;Winther.R.2006a, Arnold.D;Falk.R;Winther.R.2010a}.
In the case $n = 3$, the lowest order elements can be represented as:
\begin{diagram}
\mathbb{R} & \rightarrow & \mathcal{P}_{3}\Lambda^{0} & \xrightarrow{d} & \mathcal{P}_{2} \Lambda^{1} & \xrightarrow{d} & \mathcal{P}_{1} \Lambda^{2} & \xrightarrow{d} & \mathcal{P}_{0} \Lambda^{3} & \rightarrow & 0 \\ 
\mathbb{R} & \rightarrow & \mathcal{P}_{2}\Lambda^{0} & \xrightarrow{d} & \mathcal{P}_{1} \Lambda^{1} & \xrightarrow{d} & \mathcal{P}_{1}^{-} \Lambda^{2} & \xrightarrow{d} & \mathcal{P}_{0} \Lambda^{3} & \rightarrow & 0 \\ 
\mathbb{R} & \rightarrow & \mathcal{P}_{2}\Lambda^{0} & \xrightarrow{d} & \mathcal{P}_{2}^{-} \Lambda^{1} & \xrightarrow{d} & \mathcal{P}_{1} \Lambda^{2} & \xrightarrow{d} & \mathcal{P}_{0} \Lambda^{3} & \rightarrow & 0 \\ 
\mathbb{R} & \rightarrow & \mathcal{P}_{1}\Lambda^{0} & \xrightarrow{d} & \mathcal{P}_{1}^{-} \Lambda^{1} & \xrightarrow{d} & \mathcal{P}_{1}^{-} \Lambda^{2} & \xrightarrow{d} & \mathcal{P}_{0} \Lambda^{3} & \rightarrow & 0 \\ 
\end{diagram} 
The correspondence between the language of differential forms and classical finite element methods is summarized in Table \ref{tab:mixed-fem}. 

To link the finite element spaces, below we will require $H_{0}^{h}(\mathrm{curl}, \Omega)$, $H_{0}^{h}(\mathrm{div}, \Omega)$ and $L^{2}_{0,h}(\Omega)$ to be in the same sequence.
  \begin{table}[H]
\centering
\begin{tabular}{c  c l}
\hline\noalign{\smallskip}
$k$ & $\Lambda_{h}^{k}(\Omega)$ & Classical finite element space \\[0.5ex]
\noalign{\smallskip}\hline\noalign{\smallskip}
0 & $\mathcal{P}_{r} \Lambda^{0}(\mathcal{T})$ & Lagrange elements of degree $\leq r$ \\
1 & $\mathcal{P}_{r} \Lambda^{1}(\mathcal{T})$ & Nedelec 2nd-kind $H(\mathrm{curl})$ elements of degree $\leq r$ \\
2 & $\mathcal{P}_{r} \Lambda^{2}(\mathcal{T})$ & Nedelec 2nd-kind $H(\mathrm{div})$ elements of degree $\leq r$ \\
3 & $\mathcal{P}_{r} \Lambda^{3}(\mathcal{T})$ & discontinuous elements of degree $\leq r$ \\
\noalign{\smallskip}\hline\noalign{\smallskip}
0 & $\mathcal{P}_{r}^{-} \Lambda^{0}(\mathcal{T})$ & Lagrange elements of degree $\leq r$ \\
1 & $\mathcal{P}_{r}^{-} \Lambda^{1}(\mathcal{T})$ & Nedelec 1st-kind $H(\mathrm{curl})$ elements of order $r-1$ \\
2 & $\mathcal{P}_{r}^{-} \Lambda^{2}(\mathcal{T})$ & Nedelec 1st-kind $H(\mathrm{div})$ elements of order $r-1$ \\
3 & $\mathcal{P}_{r}^{-} \Lambda^{3}(\mathcal{T})$ & discontinuous elements of degree $\leq r-1$ \\[1ex]
\noalign{\smallskip}\hline
\end{tabular}
\caption{Correspondences between finite element differential forms and the classical finite element spaces for $n = 3$ (from \cite{Arnold.D;Falk.R;Winther.R.2006a})}\label{tab:mixed-fem}
\end{table}

As we shall see, it is useful to group the spaces to define 
$$
\bm{X}_{h}:=\bm{V}_{h}\times H^{h}_{0}(\mathrm{curl}, \Omega)\times
H^{h}_{0}(\mathrm{curl}, \Omega)\times H_{0}^{h}(\mathrm{div},
\Omega).
$$
and group $Q_{h}\times L_{0, h}^{2}(\Omega)$ to define
$$
\bm{Y}_{h}:=Q_{h}\times L_{0, h}^{2}(\Omega).
$$

For the analysis, we also need to define a reduced space, where $\bm{j}_{h}$ and $\bm{\sigma}_{h}$ (introduced below) are eliminated:
$$
\tilde{\bm{X}}_{h}:=\bm{V}_{h}\times H^{h}_{0}(\mathrm{div}, \Omega).
$$

In order to define appropriate norms, we introduce the discrete curl operator on the discrete
level. For any $\bm{C}_{h}\in H^{h}_{0}(\mathrm{div}, \Omega)$, define
$\nabla_{h}\times \bm{C}_{h}\in H^{h}_{0}(\mathrm{curl}, \Omega)$:
$$
(\nabla_{h}\times \bm{C}_{h}, \bm{F}_{h})=(\bm{C}_{h}, \nabla\times \bm{F}_{h}), \quad \forall \bm{F}_{h}\in H^{h}_{0}(\mathrm{curl},\Omega).
$$
For any $\bm{w}_{h}\in H_{0}^{h}(\mathrm{curl}, \Omega)$, we define $\nabla_{h}\cdot\bm{w}_{h}\in H_{0}^{h}(\mathrm{grad}, \Omega)$ by
$$
(\nabla_{h}\cdot \bm{w}_{h}, v_{h})=-(\bm{w}_{h}, \nabla v_{h}), \quad \forall v_{h}\in H_{0}^{h}(\mathrm{grad}, \Omega).
$$

We define $\mathbb{P}: L^{2}(\Omega)\rightarrow H^{h}_{0}(\mathrm{curl},\Omega)$ to be the $L^{2}$ projection
$$
(\mathbb{P}\phi, \bm{F}_{h})=(\phi, \bm{F}_{h}), \quad\forall \bm{F}_{h}\in H^{h}_{0}(\mathrm{curl},\Omega), \phi\in L^{2}(\Omega).
$$
We further define $\|\cdot\|_{d}$ to be a modified norm of $H^{h}_{0}(\mathrm{div}, \Omega)$ by
$$
\|\bm{C}_{h}\|_{d}^{2}:=\|\bm{C}_{h}\|^{2}+\|\nabla\cdot \bm{C}_{h}\|^{2}+\|\nabla_{h}\times \bm{C}_{h}\|^{2}.
$$ 
Moreover, $\|\cdot\|_{c}$ for $H^{h}_{0}(\mathrm{curl}, \Omega)$ is simply the $L^{2}$ norm:
$$
\|\bm{F}_{h}\|_{c}^{2}:=\|\bm{F}_{h}\|^{2}.
$$
There are some motivations to define such a stronger norm for $H_{0}^{h}(\mathrm{div}, \Omega)$ and weaker norm for $H_{0}^{h}(\mathrm{curl}, \Omega)$ space. One technical reason is that we want the nonlinear term $\nabla\times (\bm{u}_{h}\times \bm{B}_{h})$ to be bounded in some proper discretization.  But generally $\bm{u}_{h}\times \bm{B}_{h}$ may not belong to $H_{0}^{h}(\mathrm{curl})$ for $\bm{u}_{h}\in H_{0}^{1}(\Omega)^{3}$ and $\bm{B}_{h}\in H_{0}(\mathrm{div}, \Omega)$. So we choose to move the curl operator to the $H^{h}_{0}(\mathrm{div})$ test function in the variational formulation to get $(\bm{u}_{h}\times \bm{B}_{h}, \nabla_{h}\times \bm{C}_{h})$.  Therefore we add the weak curl norm to $H^{h}_{0}(\mathrm{div}, \Omega)$ space. Another motivation can be seen in the energy estimate: on the continuous level, the energy estimate contains $\bm{j}=R_{m}^{-1}\nabla\times \bm{B}$, but not $\nabla\times \bm{j}$. So it is natural to use $L^{2}$ norm for the discrete variable $\bm{j}_{h}$.

Now we define the norms for various product spaces.   For $\bm{Y}_{h}$ space,
we define
$$
\|(q,r)\|_{\bm{Y}}^{2}:=\|q\|^{2}+\|r\|^{2}.
$$
For the other product spaces,   we define 
$$
\|(\bm{u}_{h}, \bm{j}_{h}, \bm{\sigma}_{h},
\bm{B}_{h})\|_{\bm{X}}^{2}:=\|\bm{u}_{h}\|_{1}^{2}+\|\bm{j}_{h}\|^{2}_{c}+\|\bm{\sigma}_{h}\|^{2}_{c}+\|\bm{B}_{h}\|_{d}^{2},\quad (\bm{u}_{h}, \bm{j}_{h}, \bm{\sigma}_{h},
\bm{B}_{h})\in \bm{X}_{h},
$$
and
$$
\|(\bm{u}_{h}, \bm{B}_{h})\|_{\tilde{\bm{X}}}^{2}:=\|\bm{u}_{h}\|_{1}^{2}+\|\bm{B}_{h}\|_{d}^{2}, \quad (\bm{u}_{h}, \bm{B}_{h})\in \tilde{\bm{X}}_{h}.
$$

\section{Estimates for divergence-free vector fields}\label{sec:estimates}
In this section, we will establish some new regularity results for the
strong divergence-free space $H^{h}_{0}(\mathrm{div}0, \Omega)$ which
will be used for our forthcoming analysis.  The main ingredients used
in our analysis include some regularity
results for the space
$\bm{Z}:=H(\mathrm{curl}, \Omega)\cap H_{0}(\mathrm{div}0, \Omega)$
(c.f. \cite{Hiptmair.R.2002a, Schotzau.D.2004a}),  and for the space
$$
\bm{X}_{h}^{c}:=\{\bm{w}\in
H_{0}^{h}(\mathrm{curl},\Omega):\nabla_{h}\cdot\bm{w}_{h}=0 \}
$$
(c.f. \cite{Hiptmair.R.2002a, Schotzau.D.2004a}), together with some appropriately defined ``Hodge
mapping'' ($H_{d}$ below) that connects $H^{h}_{0}(\mathrm{div}0,
\Omega)$ with $\bm{Z}$.

We first give a preliminary result based on Hodge decomposition:
\begin{lemma}\label{curl-Z}
$$
\nabla\times \bm{Z}=H(\mathrm{div}0, \Omega)=\nabla\times H(\mathrm{curl}, \Omega).
$$
\end{lemma}
\begin{proof}
From the Hodge decomposition for $L^{2}(\Omega)^{3}$:
$$
L^{2}(\Omega)^{3}=\nabla H^{1}(\Omega)+\nabla\times H_{0}(\mathrm{curl}, \Omega)=H(\mathrm{curl}0, \Omega)+H_{0}(\mathrm{div}0, \Omega).
$$
Here $H(\mathrm{curl}0, \Omega):=\{\bm{F}\in H(\mathrm{curl}, \Omega): \nabla\times \bm{F}=\bm{0}  \}.$

Therefore 
\begin{align}\nonumber
H(\mathrm{curl}, \Omega)&=L^{2}(\Omega)^{3}\cap H(\mathrm{curl}, \Omega)\\&\nonumber
=H(\mathrm{curl}0, \Omega)+H_{0}(\mathrm{div}0, \Omega)\cap H(\mathrm{curl}, \Omega)\\&
=H(\mathrm{curl}0, \Omega)+\bm{Z}.
\end{align}

This implies 
$$
H(\mathrm{div}0, \Omega)=\nabla\times H(\mathrm{curl}, \Omega)=\nabla\times \bm{Z}.
$$
\end{proof}
We now define the ``Hodge mapping'' for $H^{h}_{0}(\mathrm{div}0)$ functions.  Let $H_{d}: H^{h}_{0}(\mathrm{div}0)\rightarrow \bm{Z}$ be defined by
\begin{align}\label{def:Hd}
\left(\nabla\times (H_{d}\bm{B}_{h}), \nabla\times \bm{v} \right)=\left (\nabla_{h}\times \bm{B}_{h}, \nabla\times \bm{v} \right ), \quad \forall \bm{v}\in \bm{Z}, \forall \bm{B}_{h}\in H^{h}_{0}(\mathrm{div} 0, \Omega).
\end{align}
Due to the Poincar\'{e}'s inequality of $\bm{Z}$, $\|\bm{z}\|\lesssim \|\nabla\times \bm{z}\|$ holds for any $\bm{z}\in \bm{Z}$. Therefore \eqref{def:Hd} uniquely defines $H_{d}\bm{B}_{h}$.

From Lemma \ref{curl-Z}, we have $\nabla\times \bm{Z}=H(\mathrm{div}0)$. Therefore
\begin{align}\label{def-Z-2}
\left(\nabla\times (H_{d}\bm{B}_{h}), \bm{w} \right)=\left (\nabla_{h}\times \bm{B}_{h}, \bm{w} \right ), \quad \forall \bm{w}\in H(\mathrm{div}0).
\end{align}
In particular, choosing $\bm{w}=\nabla\times (H_{d}\bm{B}_{h})$, we see
$$
\|\nabla\times (H_{d}\bm{B}_{h})\|\leq \|\nabla_{h}\times \bm{B}_{h}\|.
$$

In the following, we will use $\tilde{\bm{B}}$ to denote the continuous lifting of $\bm{B}_{h}$:
$$
\tilde{\bm{B}}:=H_{d}\bm{B}_{h}.
$$
Moreover, 
$
H_{c}: \bm{X}_{h}^{c}\rightarrow H_{0}(\mathrm{curl}, \Omega)\cap H(\mathrm{div}, \Omega)
$
 is the Hodge mapping for $H_{0}^{h}(\mathrm{curl},\Omega)$ \cite{Hiptmair.R.2002a, Schotzau.D.2004a}, defined by 
 $$
 \nabla\times (H_{c}\bm{F}_{h})=\nabla\times \bm{F}_{h}, \quad \forall \bm{F}_{h}\in \bm{X}_{h}^{c}.
$$
We also use the notation $\tilde{\bm{F}}$ to denote $H_{c}\bm{F}_{h}$ when $\bm{F}_{h}\in \bm{X}_{h}^{c}$.

\begin{lemma}[Approximation of $H_d$]\label{Hd-approximation}
 If $\Omega$ is a bounded polyhedral domain in $\mathbb{R}^3$, there exists  $0< \delta(\Omega) \leq \frac{1}{2}$ such that
$$
\|\bm{B}_{h}-{H_d}\bm{B}_h\|\lesssim h^{\frac{1}{2} + \delta}  \|\nabla_{h} \times \bm{B}_{h}\|,
$$
for all $\bm{B}_h \in H_{0}^{h}(\mathrm{div}0, \Omega)$. 
\end{lemma}
\begin{proof}

  We define $\Pi_{\mathrm{div}}^h$ to be the bounded cochain
  projection to $H^{h}_{0}(\mathrm{div}, \Omega)$ \cite{falk2014local}.  Note that $\nabla
  \cdot
  \left(\bm{B}_{h}-\Pi_{\mathrm{div}}^{h}\tilde{\bm{B}}\right)=0$ due
  to the commuting diagram. Therefore there exists $\bm{\phi}_{h}\in
  \bm{X}_{h}^{c}$ and the corresponding lifting $\tilde{\bm{\phi}}:=H_{c}\bm{\phi}_{h}\in H_{0}(\mathrm{curl}, \Omega)\cap H(\mathrm{div}, \Omega)$ such that
  $\bm{B}_{h}-\Pi_{\mathrm{div}}^{h}\tilde{\bm{B}}=\nabla \times
  \bm{\phi}_{h}=\nabla\times \tilde{\bm{\phi}}$ and  there exists  a positive constant $0< \delta(\Omega) \leq \frac{1}{2}$ such that
  \begin{equation}
    \label{phi-phi}
\|\bm{\phi}_{h}-\tilde{\bm{\phi}}\|\lesssim h^{\frac{1}{2}+\delta}\|\nabla \times \bm{\phi}_{h}\|=h^{\frac{1}{2}+\delta}\|\bm{B}_{h}-\Pi_{\mathrm{div}}^{h}\tilde{\bm{B}}\|,
  \end{equation}
  where the first inequality is from the approximation property of $H_{c}$.

From \eqref{def-Z-2}, we have
$$
(\nabla_{h}\times \bm{B}_{h}, \tilde{\bm{\phi}})=(\nabla\times \tilde{\bm{B}}, \tilde{\bm{\phi}})=(\tilde{\bm{B}}, \nabla\times \tilde{\bm{\phi}}),
$$
and
$$
(\bm{B}_{h}, \nabla\times \bm{\phi}_{h})=(\nabla_{h}\times \bm{B}_{h}, \bm{\phi}_{h})=(\nabla_{h}\times \bm{B}_{h}, \bm{\phi}_{h}-\tilde{\bm{\phi}})+(\tilde{\bm{B}}, \nabla\times \tilde{\bm{\phi}}).
$$
Namely,
$$
(\bm{B}_{h}-\tilde{\bm{B}}, \bm{B}_{h}-\Pi_{\mathrm{div}}^{h}\tilde{\bm{B}})=(\nabla_{h}\times \bm{B}_{h}, \bm{\phi}_{h}-\tilde{\bm{\phi}}).
$$
Thus
\begin{align*}
\|\bm{B}_{h}-\tilde{\bm{B}}\|^{2}&=(\bm{B}_{h}-\tilde{\bm{B}}, \bm{B}_{h}-\Pi_{\mathrm{div}}^{h}\tilde{\bm{B}})+(\bm{B}_{h}-\tilde{\bm{B}}, \Pi_{\mathrm{div}}^{h}\tilde{\bm{B}}-\tilde{\bm{B}})\\&
=(\nabla_{h}\times \bm{B}_{h}, \bm{\phi}_{h}-\tilde{\bm{\phi}})+(\bm{B}_{h}-\tilde{\bm{B}}, \Pi_{\mathrm{div}}^{h}\tilde{\bm{B}}-\tilde{\bm{B}}).
\end{align*}
By \eqref{phi-phi} and the interpolation error estimates
$$
\|\tilde{\bm{B}}-\Pi_{\mathrm{div}}^{h}\tilde{\bm{B}}\|\lesssim h^{\frac{1}{2}+\delta}\|\tilde{\bm{B}}\|_{\frac{1}{2}+\delta}\lesssim h^{\frac{1}{2}+\delta}\|\nabla\times \tilde{\bm{B}}\|\lesssim h^{\frac{1}{2}+\delta} \|\nabla_{h}\times \bm{B}_{h}\|,
$$
 we obtain
\begin{align*}
\left | (\nabla_{h}\times \bm{B}_{h}, \bm{\phi}_{h}-\tilde{\bm{\phi}}) \right |&\lesssim h^{\frac{1}{2}+\delta}\|\bm{B}_{h}-\Pi_{\mathrm{div}}^{h}\tilde{\bm{B}}\|\|\nabla_{h}\times \bm{B}_{h}\|\\&
\leq h^{\frac{1}{2}+\delta}\left (\|\bm{B}_{h}-\tilde{\bm{B}}\|  + \|\tilde{\bm{B}}-\Pi_{\mathrm{div}}^{h}\tilde{\bm{B}}\|   \right)\|\nabla_{h}\times \bm{B}_{h}\|\\&
\leq h^{\frac{1}{2}+\delta}\|\bm{B}_{h}-\tilde{\bm{B}}\|\|\nabla_{h}\times \bm{B}_{h}\|  +  h^{1+2\delta}\|\nabla_{h}\times \bm{B}_{h}\|^{2}\\&
\leq \frac{1}{2} \|\bm{B}_{h}-\tilde{\bm{B}}\|^{2}+\frac{1}{2}h^{1+2\delta}\|\nabla_{h}\times \bm{B}_{h}\|^{2}  +  h^{1+2\delta}\|\nabla_{h}\times \bm{B}_{h}\|^{2},
\end{align*}
and hence
$$
\|\bm{B}_{h}-\tilde{\bm{B}}\|^{2}\lesssim
 \|\tilde{\bm{B}}-\Pi_{\mathrm{div}}^{h}\tilde{\bm{B}}\|^{2}
+ h^{1+2\delta}\|\nabla_{h}\times \bm{B}_{h}\|^2.
$$
 
This completes the proof.
\end{proof}

For nonlinear problems and their linearizations, it is technical to prove the boundedness of variational forms, and this often requires careful estimates of regularity. The nonlinear terms in the variational forms proposed in this paper will have the form $(\bm{u}_{h}\times \bm{B}_{h}, \bm{j}_{h})$, where $\bm{u}_{h}\in \bm{V}_{h}\subset H_{0}^{1}(\Omega)^{3}$, $\bm{B}_{h}\in H^{h}_{0}(\mathrm{div}0, \Omega)$ and $\bm{j}_{h}\in H^{h}_{0}(\mathrm{curl}, \Omega)$.
\begin{lemma}\label{lem:nonlinear-boundedness}
For $\bm{u}_{h}\in \bm{V}_{h}$ and $\bm{B}_{h}\in H_{0}^{h}(\mathrm{div}0, \Omega)$, we have the following bound:
$$
\|\bm{u}_{h}\times \bm{B}_{h}\|\lesssim\|\bm{u}_{h}\|_{1}\|\nabla_{h}\times\bm{B}_{h}\|.
$$
\end{lemma}
\begin{proof}

From Lemma \ref{Hd-approximation}, we have
$$
\|\bm{B}_{h}-\tilde{\bm{B}}\|\lesssim h^{\frac{1}{2}+\delta}\|\nabla_{h}\times \bm{B}_{h}\|,
$$
where $0<\delta\leq \frac{1}{2}$ is a positive constant depending on the domain.

Then
\begin{align*}
\|\bm{u}_{h}\times \bm{B}_{h}\|\leq  \|\bm{u}_{h}\times (\bm{B}_{h}-\tilde{\bm{B}})\|+\|\bm{u}_{h}\times \tilde{\bm{B}}\|.
\end{align*}
For the first term,
\begin{align*}
\|\bm{u}_{h}\times (\bm{B}_{h}-\tilde{\bm{B}})\|&\leq \|\bm{u}_{h}\|_{0,\infty}\|\bm{B}_{h}-\tilde{\bm{B}}\|\\&
\lesssim h^{-\frac{1}{2}}\|\bm{u}_{h}\|_{0,6}\cdot h^{\frac{1}{2}}\|\nabla_{h}\times \bm{B}_{h}\|\\&
\lesssim \|\bm{u}_{h}\|_{1}\|\nabla_{h}\times \bm{B}_{h}\|,
\end{align*}
where the second inequality comes from the inverse estimates and the approximation results.

Due to the regularity of $\bm{Z}$ \cite{Hiptmair.R.2002a}, we have 
\begin{align*}
\|\bm{u}_{h}\times \tilde{\bm{B}}\|&\leq \|\bm{u}_{h}\|_{0, 6}\|\tilde{\bm{B}}\|_{0,3}\\&
\lesssim  \|\bm{u}_{h}\|_{1}\|\nabla\times \tilde{\bm{B}}\|\\&
\leq  \|\bm{u}_{h}\|_{1}\|\nabla_{h}\times \bm{B}_{h}\|.
\end{align*}
This implies
$$
\|\bm{u}_{h}\times \bm{B}_{h}\|\lesssim \|\bm{u}_{h}\|_{1}\|\nabla_{h}\times \bm{B}_{h}\|.
$$
\end{proof}

Below we will use a positive constant $C_{2}$ to denote the bound:
\begin{align}\label{def:C2}
\|\bm{u}_{h}\times \bm{B}_{h}\|\leq C_{2}\|\nabla\bm{u}_{h}\|\|\nabla_{h}\times\bm{B}_{h}\|,
\end{align}
and therefore
$$
(\bm{u}_{h}\times \bm{B}_{h}, \bm{j}_{h})\leq C_{2}\|\nabla\bm{u}_{h}\|\|\nabla_{h}\times\bm{B}_{h}\|\|\bm{j}_{h}\|_{0}.
$$

In the discussions below, we will need discrete Poincar\'{e}'s inequality for $H^{h}_{0}(\mathrm{div}0, \Omega)$ functions. We note that
 the two dimensional case is given in \cite{Chen.L;Wang.M;Zhong.L.2014a}, and the proof can be modified to adapt to the three dimensional case. We include a different proof here.
 \begin{lemma}\label{lem:poincare}
For $\bm{B}_{h}\in H^{h}_{0}(\mathrm{div}0, \Omega)$, we have the following discrete Poincar\'{e}'s inequality:
$$
\|\bm{B}_{h}\|\lesssim \|\nabla_{h}\times\bm{B}_{h}\|.
$$
\end{lemma}
\begin{proof}
Because $\nabla\cdot\bm{B}_{h}=0$, we can choose $\bm{E}_{h}\in H^{h}_{0}(\mathrm{curl})$ such that
$$
\nabla\times \bm{E}_{h}=\bm{B}_{h}, \quad \mbox{ and } \quad \nabla_{h}\cdot\bm{E}_{h}=0.
$$
From the discrete Poincar\'{e} inequality for $\bm{X}_{h}^{c}$ in \cite{Arnold.D;Falk.R;Winther.R.2010a}, 
\begin{align}\label{poincare0}
\|\bm{E}_{h}\|_{\mathrm{curl}}\lesssim \|\nabla\times \bm{E}_{h}\|=\|\bm{B}_{h}\|.
\end{align}

We have
\begin{align}\nonumber
\|\nabla_{h}\times \bm{B}_{h}\|&=\sup_{\bm{F}_{h}\in H^{h}_{0}(\mathrm{curl})}\frac{(\nabla_{h}\times\bm{B}_{h}, \bm{F}_{h})}{\|\bm{F}_{h}\|}\\&\label{poincare-l2dual}
=\sup_{\bm{F}_{h}\in H^{h}_{0}(\mathrm{curl})}\frac{(\bm{B}_{h}, \nabla\times\bm{F}_{h})}{\|\bm{F}_{h}\|}.
\end{align}
Therefore combining \eqref{poincare-l2dual} and \eqref{poincare0}, we get
$$
\|\nabla_{h}\times \bm{B}_{h}\|\geq \frac{(\bm{B}_{h}, \nabla\times \bm{E}_{h})}{\|\bm{E}_{h}\|},
$$
and
$$
\|\nabla_{h}\times \bm{B}_{h}\|\gtrsim \|\bm{B}_{h}\|.
$$
\end{proof}

Combined with $L^{p}$-$L^{p}$ bounded interpolations (c.f. \cite{christiansen2011topics}), we can further establish $L^{p}$ estimates of $H(\mathrm{div}0)$ finite element functions.
\begin{theorem}\label{thm:L3}
For bounded Lipschitz polyhedral domain $\Omega$,  we have
$$
\|\bm{B}_{h}\|_{0,3}\lesssim \|\nabla_{h}\times \bm{B}_h\|,\quad \bm{B}_{h}\in  H_{0}^{h}(\mathrm{div}0, \Omega).
$$
\end{theorem}
\begin{proof}
From triangular inequality,  we have
$$
\|\bm{B}_{h}\|_{0,3}\leq \|\bm{B}_{h}-\Pi_{\mathrm{div}}^{h}{H_{d}}{\bm{B}}\|_{0,3}+\|\Pi_{\mathrm{div}}^{h}{{H}_{d}}{\bm{B}}\|_{0,3}.
$$
From inverse estimates, interpolation error estimates and the approximation of Hodge mapping (Lemma \ref{Hd-approximation}), 
\begin{align*}
\|\bm{B}_{h}-\Pi_{\mathrm{div}}^{h}{{H}_{d}}{\bm{B}_{h}}\|_{0,3}&\lesssim h^{-1/2}\left \|\bm{B}_{h}-\Pi_{\mathrm{div}}^{h}{{H}_{d}}{\bm{B}_{h}}\right \|\\&
\lesssim h^{-1/2}( \left \|\bm{B}_{h}-H_{d}{\bm{B}_{h}}\right \|+\left \|H_{d}{\bm{B}_{h}}-\Pi_{\mathrm{div}}^{h}{{H}_{d}}{\bm{B}_{h}}\right \|   )\\&
\lesssim \|\nabla_{h}\times \bm{B}_{h}\|.
\end{align*}
Using the $L^{3}$ stability of the interpolation operator and regularity results of $\bm{Z}$, we have
$$
\left \|\Pi_{\mathrm{div}}^{h}{{H}_{d}}{\bm{B}_{h}}\right \|_{0,3}\lesssim\left \|{{H}_{d}}{\bm{B}_{h}}\right \|_{0, 3}\lesssim \left \|\nabla\times {{H}_{d}}{\bm{B}_{h}}\right \| \leq \|\nabla_{h}\times \bm{B}_{h}\|.
$$
Then the triangular inequality implies 
$$
\|\bm{B}_{h}\|_{0,3}\leq \|\bm{B}_{h}-\Pi_{\mathrm{div}}^{h}{{H}_{d}}{\bm{B}_{h}}\|_{0,3}+\left \|\Pi_{\mathrm{div}}^{h}{{H}_{d}}{\bm{B}_{h}}\right \|_{0,3}  \lesssim \|\nabla_{h}\times \bm{B}_h\|.
$$
\end{proof}

In following discussions, we still use a generic constant $C_{2}$ to denote the bound
$$
\|\bm{B}_{h}\|\leq C_{2}\|\nabla_{h}\times \bm{B}_{h}\|, \quad\forall \bm{B}_{h}\in H^{h}_{0}(\mathrm{div}, \Omega).
$$

\section{A new finite element formulation}\label{sec:Bj}
In Hu, Ma and Xu \cite{hu2014stable}, the authors studied a numerical scheme using $\bm{B}$ and $\bm{E}$ as variables.
A straightforward analysis by Brezzi theory leads to a stringent condition on the time step size.
In this section, we propose a new finite element
scheme whose well-posedness will not depend on such assumptions. 

We note that it is the variable $\bm{j}$ that appears in the energy
estimate. Therefore it seems natural to use $\bm{B}$ and $\bm{j}$ as
mixed variables of the electromagnetic part of the MHD system. Discretization methods based on $\bm{B}$ and $\bm{j}$ actually have
already existed in the literature.  For example, some finite volume
methods using $\bm{B}$ and $\bm{j}$  have been developed in
\cite{ Ni.M;Munipalli.R;Morley.N;Huang.P;Abdou.M.2007a, Ni.M;Li.J.2012a} where the conservation of
$\nabla\cdot \bm j=0$ was considered (but no discussion on the
condition $\nabla \cdot \bm B=0$), and in \cite{Yang.Z;Zhou.T;Chen.H;Ni.M.2010a}, 
$\bm{B}$ and $\bm{j}$ were used as variables in the simulation of liquid metal
breeder blankets.

We eliminate $\bm{E}$ by Ohm's law and consider the following model: 
\begin{dgroup}[compact]\label{jmhd}
\begin{dmath}
-R_{e}^{-1}\Delta \bm{u}+\nabla p +(\bm{u}\cdot \nabla)\bm{u}+S\bm{B}\times \bm{j}=\bm{f}, \label{jmhd1}
\end{dmath}
\begin{dmath}
\nabla\times \bm{j}-\nabla\times (\bm{u}\times \bm{B})=\bm{0}, \label{jmhd2}
\end{dmath}
\begin{dmath}
\bm{j}-R_{m}^{-1}\nabla\times \bm{B}=\bm{0}, \label{jmhd3}
\end{dmath}
\begin{dmath}
\nabla\cdot \bm{u}=0, \label{jmhd4}
\end{dmath}
\begin{dmath}
\nabla\cdot \bm{B}=0. \label{jmhd5}
\end{dmath}
\end{dgroup}

The well-posedness of the continuous formulation has been shown in \cite{Schotzau.D.2004a}.  The author proved that there exists at least one solution $\bm{u}\in H_{0}^{1}(\Omega)^{3}$, $\bm{B}\in H(\mathrm{curl}, \Omega)\cap H_{0}(\mathrm{div}0, \Omega)$ for the nonlinear system where $\bm{j}$ is eliminated. The variational form reads:
find $\left (\bm{u}, \bm{B}, p, \phi \right )\in {H}_{0}^{1}(\Omega)^{3}\times  H(\mathrm{curl}, \Omega)\cap H_{0}(\mathrm{div}, \Omega)\times L^{2}_{0}(\Omega)\times H^{1}_{0}(\Omega)$ such that for any $\left (\bm{v}, \bm{C}, q, \psi \right )\in {H}_{0}^{1}(\Omega)^{3}\times  H(\mathrm{curl}, \Omega)\cap H_{0}(\mathrm{div}, \Omega)\times L^{2}_{0}(\Omega)\times H^{1}_{0}(\Omega)$,
\begin{align}\label{curl-formulation}
\begin{cases}
&L(\bm{u}; \bm{u}, \bm{v})+R_{e}^{-1}(\nabla\bm{u}, \nabla\bm{v})-s\left ((\nabla\times \bm{B})\times \bm{B}, \bm{v}  \right )-(p, \nabla\cdot\bm{v})=\langle \bm{f}, \bm{v}\rangle\\& -(\bm{u}\times \bm{B}, \nabla\times \bm{C})+R_{m}^{-1}(\nabla\times \bm{B}, \nabla\times \bm{C})+(\nabla \phi, \bm{C})=0, \\&
(\nabla\cdot\bm{u}, q)=0,\\
&(\bm{B}, \nabla \psi)=0.
\end{cases}
\end{align}
Considering $\bm{j}=R_{m}^{-1}\nabla\times \bm{B}$  as an intermediate variable, we conclude with the existence of  solutions to \eqref{jmhd}: for any $\bm{f}\in \left (H^{1}_{0}(\Omega)^{3}\right )^{\ast}$, there exists at least one solution  $\bm{u}\in H_{0}^{1}(\Omega)^{3}$, $\bm{B}\in H(\mathrm{curl}, \Omega)\cap H_{0}(\mathrm{div}0, \Omega)$ and $\bm{j}\in L^{2}(\Omega)^{3}$.

\subsection{Mixed finite element discretizations}

We now present our new finite element discretization of the
above system \eqref{jmhd}. 
\begin{problem}\label{prob:nonlinear}
Given $\bm{f}\in \bm{V}_{h}^{\ast}$. Find $(\bm{u}_{h},\bm{j}_{h}, \bm{\sigma}_{h}, \bm{B}_{h}, p_{h}, r_{h})\in \bm{X}_{h}\times \bm{Y}_{h} $, such that for any $(\bm{v}_{h},  \bm{k}_{h}, \bm{\tau}_{h}, \bm{C}_{h},  q_{h},s_{h})\in \bm{X}_{h}\times \bm{Y}_{h} $,
\begin{dgroup}[compact]\label{equi-fem}
\begin{dmath}
R_{e}^{-1}(\nabla \bm{u}_{h}, \nabla \bm{v}_{h})+L(\bm{u}_{h}; \bm{u}_{h}, \bm{v}_{h})  +S(  \bm{j}_{h}, \bm{v}_{h}\times \bm{B}_{h})-(p_{h}, \nabla\cdot \bm{v}_{h})
 =\langle \bm{f}, \bm{v}_{h} \rangle, \label{equi-fem1}
\end{dmath}
\begin{dmath}
SR_{m}^{-1}(\nabla\times \bm{j}_{h},  \bm{C}_{h})-SR_{m}^{-1}(\nabla\times \bm{\sigma}_{h},   \bm{C}_{h})+({r}_{h}, \nabla\cdot \bm{C}_{h})=0, \label{equi-fem2}
\end{dmath}
\begin{dmath}
SR_{m}^{-1}(\bm{\sigma}_{h}, \bm{\tau}_{h})-SR_{m}^{-1}(\bm{u}_{h}\times \bm{B}_{h}, \bm{\tau}_{h})=0, \label{equi-fem3}
\end{dmath}
\begin{dmath}
{S}(\bm{j}_{h}, \bm{k}_{h})-SR_{m}^{-1}(\bm{B}_{h}, \nabla\times \bm{k}_{h})=0, \label{equi-fem4}
\end{dmath}
\begin{dmath}
-(\nabla\cdot \bm{u}_{h}, q_{h})=0, \label{equi-fem5}
\end{dmath}
\begin{dmath}
(\nabla\cdot \bm{B}_{h}, s_{h})=0. \label{equi-fem6}
\end{dmath}
\end{dgroup}
\end{problem}

In the above scheme, an additional variable $\bm{\sigma}_h$ is
introduced to accommodate for the evaluation of the discrete curl operator
$\nabla_h\times$ which is nonlocal.  This extra work comes
from the nonlinear coupling term $(\nabla\times (\bm{u}\times \bm{B}),
\bm{C})$, because curl operator cannot act on $\bm{u}\times \bm{B}$ directly.

Before further discussions, we verify basic properties of the discretization and the energy estimates, which are basic and important tools in the design and  analysis of numerical methods, especially for nonlinear problems.
\begin{theorem}\label{them:nonlinear-div-free}
Any solution of Problem \ref{prob:nonlinear} satisfies 
\begin{enumerate}
\item
Gauss's law of magnetic field in the strong sense:
$$
\nabla\cdot \bm{B}_{h}=0.
$$
\item
the Lagrange multiplier $r_{h}=0$, hence \eqref{equi-fem2} reduces to 
$$
\nabla\times (\bm{j}_{h}- \bm{\sigma}_{h})=\bm{0}.
$$
\item
the energy estimates:
$$
R_{e}^{-1}\|\nabla\bm{u}_{h}\|^{2}+S\|\bm{j}_{h}\|^{2}=\langle \bm{f}, \bm{u}\rangle,
$$
and
$$
\frac{1}{2R_{e}}\|\nabla\bm{u}_{h}\|^{2}+{S}\|\bm{j}_{h}\|^{2}\leq \frac{R_{e}}{2}\|\bm{f}\|_{-1}^{2}.
$$
\end{enumerate}
\end{theorem}

\begin{proof}
\begin{enumerate}
\item
It is a direct consequence of \eqref{equi-fem6}, since $\nabla\cdot H^{h}_{0}(\mathrm{div};\Omega) = L^{2}_{0,h}(\Omega)$.
\item
Take $\bm{C}_{h}=\nabla\times \bm{j}_{h}-\nabla\times \bm{\sigma}_{h}$. From \eqref{equi-fem2} we see
$$
\nabla\times \bm{j}_{h}-\nabla\times \bm{\sigma}_{h}=\bm{0}.
$$
Hence
$$
(r_{h}, \nabla\cdot \bm{C}_{h})=0, \quad \forall \bm{C}_{h}\in H^{h}_{0}(\mathrm{div},\Omega).
$$
This implies
$$
r_{h}=0.
$$
\item
Take $\bm{v}_{h}=\bm{u}_{h}, \bm{C}_{h}=\bm{B}_{h}, \bm{\tau}_{h}=\nabla_{h}\times \bm{B}_{h}, \bm{k}_{h}=\bm{j}_{h}$ in  \eqref{equi-fem1}-\eqref{equi-fem4}. Add them together, we have
$$
R_{e}^{-1}\|\nabla\bm{u}_{h}\|^{2}+S\|\bm{j}_{h}\|^{2}+S(\bm{j}_{h}, \bm{u}_{h}\times \bm{B}_{h})-SR_{m}^{-1}(\bm{u}_{h}\times \bm{B}_{h}, \nabla_{h}\times \bm{B}_{h})=\langle \bm{f}, \bm{u}_{h}\rangle.
$$ 
Again from \eqref{equi-fem4}, the last two terms on the left hand side vanish by taking $\bm{k}_{h}=\mathbb{P}(\bm{u}_{h}\times \bm{B}_{h})$.

This implies the desired result.
\end{enumerate}
\end{proof}

From \eqref{equi-fem3}, we see $\bm{\sigma}_{h}=\mathbb{P}(\bm{u}_{h}\times \bm{B}_{h})$, and from \eqref{equi-fem4}, we get $\bm{j}_{h}={R_m^{-1}}\nabla_{h}\times 
\bm{B}_{h}$.  To prove the existence of solution of the nonlinear scheme, we formally eliminate $\bm{\sigma}_{h}$ and $\bm{j}_{h}$ using the above identities, to get a system with $\bm{u}_{h}$, $\bm{B}_{h}$, $p_{h}$ and $s_{h}$.

For this purpose, we define 
\begin{align*}
\tilde{\bm{a}}(\tilde{\bm{\psi}}_{h}; \tilde{\bm{\xi}}_{h},& \tilde{\bm{\eta}}_{h}):= R_{e}^{-1}(\nabla \bm{u}_{h}, \nabla \bm{v}_{h}) +L(\bm{w}_{h}; \bm{u}_{h}, \bm{v}_{h})   \\
&+SR_{m}^{-1} (\nabla_{h}\times\bm{B}_{h}, \bm{v}_{h} \times \bm{G}_{h} )- SR_{m}^{-1}( \bm{u}_{h}\times \bm{G}_{h}, \nabla_{h}\times \bm{C}_{h})\\
&+ SR_{m}^{-2}( \nabla_{h}\times \bm{B}_{h}, \nabla_{h}\times \bm{C}_{h}),
\end{align*}
and 
$$
\bm{b}(\tilde{\bm{\xi}}_{h}, \bm{y}_{h}):= -(\nabla\cdot \bm{u}_{h}, q_{h})+(\nabla\cdot \bm{B}_{h}, s_{h}).
$$
Hereafter, $\tilde{\bm{\psi}}_{h}$, $\tilde{\bm{\xi}}_{h}$, $\tilde{\bm{\eta}}_{h}$, $\bm{y}_{h}$ are short for $(\bm{w}_{h},  \bm{G}_{h})$, $(\bm{u}_{h},  \bm{B}_{h})$, $(\bm{v}_{h},  \bm{C}_{h})\in \tilde{\bm{X}}_{h}$ and $\left ( q_{h}, s_{h}\right )\in \bm{Y}_{h}$.

Eliminating $\bm{j}_{h}$ and $\bm{\sigma}_{h}$, Problem \ref{prob:nonlinear} is equivalent to the following form.
\begin{problem} \label{prob:nonlinear-equi}
Given  $\tilde{\bm{\theta}}=(\bm{f}, \bm{0})\in \tilde{\bm{X}}_{h}^{\ast}$,   find $\tilde{\bm{\xi}}_{h}\in \tilde{\bm{X}}_{h}$, $\bm{x}_{h}\in \bm{Y}_{h}$, such that
\begin{align}
\tilde{\bm{a}}(\tilde{\bm{\xi}}_{h}; \tilde{\bm{\xi}}_{h}, \tilde{\bm{\eta}}_{h})+ {\bm{b}}(\tilde{\bm{\eta}}_{h}, \bm{x}_{h})&=\langle \tilde{\bm{\theta}}, \tilde{\bm{\eta}}_{h} \rangle ,\quad \forall~ \tilde{\bm{\eta}}_{h} \in \tilde{\bm{X}}_{h}, \\
{\bm{b}}(\tilde{\bm{\xi}}_{h},\bm{y}_{h})&=0, \quad\forall \bm{y}_{h}\in \bm{Y}_{h}.
\end{align}
where $\langle \tilde{\bm{\theta}}, \tilde{\bm{\eta}}_{h} \rangle:=\langle \bm{f}, \bm{v}_{h}\rangle$.
\end{problem}

To see the equivalence, we note that if $(\bm{u}_{h},\bm{j}_{h}, \bm{\sigma}_{h}, \bm{B}_{h}, p_{h}, r_{h})\in \bm{X}_{h}\times \bm{Y}_{h} $ solves Problem \ref{prob:nonlinear}, then $(\bm{u}_{h}, \bm{B}_{h}, p_{h}, r_{h})\in \tilde{\bm{X}}_{h}\times \bm{Y}_{h}$ solves Problem \ref{prob:nonlinear-equi} with the same data and $\|(\bm{u}_{h}, \bm{B}_{h})\|_{\tilde{\bm{X}}}\leq  \|(\bm{u}_{h},\bm{j}_{h}, \bm{\sigma}_{h}, \bm{B}_{h})\|_{\bm{X}}  $. Conversely, from a solution $(\bm{u}_{h}, \bm{B}_{h}, p_{h}, r_{h})$ of Problem \ref{prob:nonlinear-equi}, we can reconstruct $\left(\bm{u}_{h},\nabla_{h}\times \bm{B}_{h}, \mathbb{P}(\bm{u}_{h}\times \bm{B}_{h}), \bm{B}_{h}, p_{h}, r_{h}\right)\in \bm{X}_{h}\times \bm{Y}_{h} $ which solves Problem \ref{prob:nonlinear} with the same data, and
$$
\left\|\left(\bm{u}_{h},\nabla_{h}\times \bm{B}_{h}, \mathbb{P}(\bm{u}_{h}\times \bm{B}_{h}), \bm{B}_{h}\right)\right\|_{{\bm{X}}}\leq 2\|(\bm{u}_{h}, \bm{B}_{h})\|_{\tilde{\bm{X}}}.
$$

Such a variational form is closely related to the ``curl-formulation'', for example, in \cite{Schotzau.D.2004a}. Here curl operators are replaced by its discrete version ``$\nabla_{h}\times$''.

The existence of solution of the nonlinear discrete scheme \eqref{equi-fem} can be stated as
\begin{theorem}\label{thm:existence-nonlinear}
There exists at least one solution $(\bm{u}_{h}, \bm{B}_{h}, p_{h}, r_{h})\in \tilde{\bm{X}}_{h}\times \bm{Y}_{h}$ solving Problem \ref{prob:nonlinear-equi}.  Therefore there exists at least one solution $(\bm{u}_{h},\bm{j}_{h}, \bm{\sigma}_{h}, \bm{B}_{h}, p_{h}, r_{h})\in \bm{X}_{h}\times \bm{Y}_{h}$ solving Problem \ref{prob:nonlinear}.
\end{theorem}

It suffices to prove the existence of solution of Problem \ref{prob:nonlinear-equi} under the norm
\begin{align}\label{norm}
\|(\bm{u}_{h}, \bm{B}_{h}, p_{h}, r_{h})\|_{A}^{2}:=\|\bm{u}_{h}\|_{1}^{2}+\|\bm{B}_{h}\|_{d}^{2}+\|p_{h}\|^{2}+\|r_{h}\|^{2}.
\end{align}

Define the kernel space ${\tilde{\bm{X}}_{h}^{0}}$ by
$$
{\tilde{\bm{X}}_{h}^{0}}:=\{ \tilde{\bm{\eta}}_{h}\in \tilde{\bm{X}}_{h}: {\bm{b}}(\tilde{\bm{\eta}}_{h}, \bm{y}_{h})=0, ~\forall \bm{y}_{h}\in \bm{Y}_{h}   \}.
$$
Following a general routine of Brezzi theory, we first establish boundedness and inf-sup conditions of the variational form.
\begin{lemma} {\bf (Boundedness)  }  \label{lem:nonlinear-boundedness2}  
With the norms given in \eqref{norm},
 ${\bm{b}}(\cdot,\cdot)$ is bounded and 
  $\tilde{\bm{a}}(\cdot; \cdot,\cdot)$ is bounded in ${\tilde{\bm{X}}_{h}^{0}}$.
 \end{lemma}
 From the construction of solutions in Brezzi theory, we note that it is enough to prove the boundedness of $\tilde{\bm{a}}(\cdot; \cdot,\cdot)$ in ${\tilde{\bm{X}}_{h}^{0}}$.
\begin{proof}
From Cauchy inequality and imbedding theorem,
$$
\left((\bm{u}_{h}\cdot \nabla)\bm{u}_{h}, \bm{v}_{h}\right)\leq  \|\bm{u}_{h}\|_{0,3}\|\nabla\bm{u}_{h}\|\|\bm{v}_{h}\|_{0,6}\lesssim \|\bm{u}_{h}\|_{1}\| \|\bm{u}_{h}\|_{1}\|\bm{v}_{h}\|_{1}.
$$
Similarly,
 $$
  ((\bm{u}_{h}\cdot \nabla)\bm{v}_{h}, \bm{u}_{h})\lesssim \|\bm{u}_{h}\|_{1}\| \|\bm{u}_{h}\|_{1}\|\bm{v}_{h}\|_{1}.
  $$
Furthermore, from Lemma \ref{lem:nonlinear-boundedness},
  $$
  (  \bm{j}_{h}, \bm{v}_{h}\times \bm{B}_{h})\lesssim  \|\nabla_{h}\times\bm{B}_{h}\|\|\bm{j}_{h}\|_{c}\|\bm{v}_{h}\|_{1}\leq \|\bm{B}_{h}\|_{d}\|\bm{j}_{h}\|_{c}\|\bm{v}_{h}\|_{1},
  $$
  and
  $$
 (\bm{u}_{h}\times \bm{B}_{h},  \nabla_{h}\times \bm{C}_{h})\lesssim \|\nabla_{h}\times\bm{B}_{h}\| \|\bm{u}_{h}\|_{1}\| \bm{C}_{h}\|_{d}\leq \|\bm{B}_{h}\|_{d} \|\bm{u}_{h}\|_{1}\| \bm{C}_{h}\|_{d}.
  $$
  The boundedness of other linear terms are obvious.
\end{proof}

Here we note again that the estimate of the boundedness of $(\bm{u}_{h}\times \bm{B}_{h}, \nabla_{h}\times \bm{C}_{h})$ is a major motivation of introducing the modified $\|\cdot\|_{c}$ and $\|\cdot\|_{d}$ norms, because $\bm{u}_{h}\times \bm{B}_{h}$ may not be in $H^{h}(\mathrm{curl}, \Omega)$, so $\mathrm{curl}$ is actually a discrete operator acting on the $H^{h}_{0}(\mathrm{div})$ function $\bm{B}_{h}$.

\begin{lemma}{\bf (inf-sup condition of ${\bm{b}}(\cdot,\cdot)$)} \label{lem:inf-supb}
There exists a positive constant $\alpha$ such that
$$
\inf_{\bm{y}_{h}\in \bm{Y}_{h}}\sup_{\tilde{\bm{\eta}}_{h}\in \tilde{\bm{X}}_{h}} \frac{{\bm{b}}( \tilde{\bm{\eta}}_{h}, \bm{y}_{h})}{\|\tilde{\bm{\eta}}_{h}\|_{\tilde{\bm{X}}}\|\bm{y}_{h}\|_{\bm{Y}}}\geq \alpha>0.
$$
\end{lemma}
\begin{proof}
It suffices to prove the following two inf-sup conditions of the pressure and magnetic multipliers: there exists constant $\alpha_{0}>0$ such that
$$
\inf_{q_{h}\in Q_{h}}\sup_{\bm{v}_{h}\in \bm{V}_{h}} \frac{(\nabla\cdot \bm{v}_{h}, q_{h})}{\|\bm{v}_{h}\|_{1}\|q_{h}\|}\geq \alpha_{0}>0,
$$
$$
\inf_{s_{h}\in L_{0,h}^{2}(\Omega)}\sup_{\bm{C}_{h}\in H^{h}_{0}(\mathrm{div};\Omega)} \frac{(\nabla\cdot \bm{C}_{h}, s_{h})}{\|\bm{C}_{h}\|_{d}\|s_{h}\|}\geq \alpha_{0}>0.
$$

The first inequality is standard for existing Stokes pairs. Now we focus on the second. The proof is a three dimensional case of the discussion in Chen et al. \cite{Chen.L;Wang.M;Zhong.L.2014a}. We include the proof here for completeness. The major difficulty is that $\|\cdot\|_{d}$ is a stronger norm than $\|\cdot\|_{\mathrm{div}}$.

It is known that for any $s_{h}\in L_{0,h}^{2}(\Omega)$, there exists $\bm{v}\in H^{1}_{0}(\Omega)^{3}$, such that
$$
\nabla\cdot \bm{v}=s_{h},
$$
and
$$
\|\bm{v}\|_{1}\lesssim \|s_{h}\|.
$$
Let $\Pi^{\mathrm{div}}$ and $\Pi^{0}$  be the interpolation in $H_{0}^{h}(\mathrm{div}, \Omega)$ and $L_{0,h}^{2}$ (we refer to \cite{Boffi.D;Brezzi.F;Fortin.M.2013a} for the definition, and \cite{falk2014local} for the local bounded cochain projection, which is bounded in $H(\mathrm{curl})$ and $H(\mathrm{div})$). We denote $\bm{v}_{h}=\Pi^{\mathrm{div}}\bm{v}$.
Then
$$
\nabla\cdot \bm{v}_{h}=\nabla\cdot \Pi^{\mathrm{div}}\bm{v}=\Pi^{0}\nabla\cdot \bm{v}=\Pi^{0}s_{h}=s_{h}.
$$
Note that $\bm{v}\in H_{0}^{1}(\Omega)^{3}$, hence $\Pi^{\mathrm{div}}$ is well-defined and bounded. Therefore 
$$
\|\bm{v}_{h}\|=\|\Pi^{\mathrm{div}}\bm{v}\|\leq \|\Pi^{\mathrm{div}}\|\|\bm{v}\|_{1}\lesssim \|\Pi^{\mathrm{div}}\|\|s_{h}\|. 
$$

Now it suffices to prove $\|\nabla_{h}\times \bm{v}_{h}\|\lesssim \|s_{h}\|$.

In fact, using inverse inequality and approximation results (see, for example, \cite{brenner2008mathematical} and \cite{Boffi.D;Brezzi.F;Fortin.M.2013a}), 
\begin{align*}
(\nabla_{h}\times \bm{v}_{h}, \nabla_{h}\times \bm{v}_{h} )&=(\nabla_{h}\times \bm{v}_{h}-\nabla\times \bm{v}, \nabla_{h}\times \bm{v}_{h} )+(\nabla\times \bm{v}, \nabla_{h}\times \bm{v}_{h} )\\
&=(\bm{v}_{h}- \bm{v},\nabla\times  \nabla_{h}\times \bm{v}_{h} )+(\nabla\times \bm{v}, \nabla_{h}\times \bm{v}_{h} )\\
&\lesssim h^{-1}\|\bm{v}_{h}- \bm{v}\| \| \nabla_{h}\times \bm{v}_{h} \|+\|\bm{v}\|_{1}\|\nabla_{h}\times \bm{v}_{h} \|\\
& \lesssim \|\bm{v}\|_{1}\|\nabla_{h}\times \bm{v}_{h} \|\\
& \lesssim \|s_{h}\|\|\nabla_{h}\times \bm{v}_{h} \|.
\end{align*}
Therefore
$$
\|\nabla_{h}\times \bm{v}_{h}\|\lesssim \|s_{h}\|.
$$
This proves the desired result.

\end{proof}

Next we consider to solve the subsystem related to $\tilde{\bm{a}}(\cdot;\cdot,\cdot)$. We introduce the existence theorem for nonlinear variational forms, which is given in, for example, \cite{Girault.V;Raviart.P.1986a}. Since we focus on the discrete level here, we only given the results for finite dimensional problems.
\begin{theorem}\label{thm:nonlinear-monotone}
Assume that the dimension of $V$ is finite, and there exists a positive constant $\alpha$ such that bounded trilinear form $a(\cdot;\cdot,\cdot)$ on $\bm{V}$ satisfies
$$
a(\bm{v}; \bm{v}, \bm{v})\geq \alpha \|\bm{v}\|^{2}, \quad \forall \bm{v}\in \bm{V}.
$$
Then the problem: given $\bm{f}\in \bm{V}^{\ast}$, find $\bm{u}\in \bm{V}$, such that for all $\bm{v}\in \bm{V}$, 
$$
a(\bm{u}; \bm{u}, \bm{v})=\bm{f}(\bm{v}),
$$ 
has at least one solution.
\end{theorem}

It is easy to see that
$$
\tilde{\bm{a}}(\tilde{\bm{\xi}}_{h}; \tilde{\bm{\xi}}_{h}, \tilde{\bm{\xi}}_{h})=R_{e}^{-1}\|\nabla\bm{u}\|^{2}+\|\nabla_{h}\times \bm{B}\|^{2}.
$$
From the discrete Poincar\'{e} inequality (Lemma \ref{lem:poincare}), we have $\tilde{\bm{a}}(\tilde{\bm{\xi}}_{h}; \tilde{\bm{\xi}}_{h}, \tilde{\bm{\xi}}_{h})\gtrsim \|\tilde{\bm{\xi}}_{h}\|_{\tilde{\bm{X}}}^{2}$ on $\tilde{\bm{X}}_{h}^{0}$. Therefore the condition in Theorem \ref{thm:nonlinear-monotone} is satisfied with $V={\tilde{\bm{X}}_{h}^{0}}$ and $a(\cdot;\cdot,\cdot)=\tilde{\bm{a}}(\cdot;\cdot,\cdot)$.

Combining Theorem \ref{thm:nonlinear-monotone} with the boundedness (Lemma \ref{lem:nonlinear-boundedness2}) and the inf-sup condition of $\bm{b}(\cdot, \cdot)$ (Lemma \ref{lem:inf-supb}), we have proved the existence of solution of nonlinear discrete problem (Theorem \ref{thm:existence-nonlinear}).

\subsection{Picard iterations}

In order to solve nonlinear Problem \ref{prob:nonlinear}, the following Picard iteration can be used:

\begin{algorithm}\label{alg:picardj} 
For $n=1,2,3,\dots$,  given $(\bm{u}_{h}^{n-1},  \bm{B}_{h}^{n-1})\in \bm{V}_{h}\times H_{0}^{h}(\mathrm{div}, \Omega)$, $\bm{f}\in \bm{V}_{h}^{\ast}$. Find $(\bm{u}_{h}^{n},\bm{j}_{h}^{n}, \bm{\sigma}^{n}_{h}, \bm{B}_{h}^{n}, p_{h}^{n}, r_{h}^{n})\in \bm{X}_{h}\times \bm{Y}_{h}$, such that for any $(\bm{v}_{h},  \bm{k}_{h}, \bm{\tau}_{h},\bm{C}_{h},  q_{h},s_{h})\in \bm{X}_{h}\times \bm{Y}_{h}$,
\begin{dgroup}[compact]
\begin{dmath}
R_{e}^{-1}(\nabla \bm{u}_{h}^{n}, \nabla \bm{v}_{h})
+L(\bm{u}^{n-1}_{h}; \bm{u}^{n}_{h}, \bm{v}_{h})  +S(  \bm{j}^{n}_{h}, \bm{v}_{h}\times \bm{B}^{n-1}_{h})- (p^{n}_{h}, \nabla\cdot \bm{v}_{h})=\langle \bm{f}, \bm{v}_{h} \rangle,
\label{equi-picardj1}
\end{dmath}
\begin{dmath}
SR_{m}^{-1}(\nabla\times \bm{j}^{n}_{h},  \bm{C}_{h})-SR_{m}^{-1}(\nabla\times \bm{\sigma}^{n}_{h},   \bm{C}_{h})+({r}^{n}_{h}, \nabla\cdot \bm{C}_{h})=0, 
\label{equi-picardj2}
\end{dmath}
\begin{dmath}
SR_{m}^{-1}(\bm{\sigma}^{n}_{h}, \bm{\tau}_{h})-SR_{m}^{-1}(\bm{u}^{n}_{h}\times \bm{B}^{n-1}_{h}, \bm{\tau}_{h})=0,
\label{equi-picardj3}
\end{dmath}
\begin{dmath}
{S}(\bm{j}^{n}_{h}, \bm{k}_{h})-SR_{m}^{-1}(\bm{B}^{n}_{h}, \nabla\times \bm{k}_{h})=0,
\label{equi-picardj4}
\end{dmath}
\begin{dmath}
-(\nabla\cdot \bm{u}^{n}_{h}, q_{h})=0,
\label{equi-picardj5}
\end{dmath}
\begin{dmath}
(\nabla\cdot \bm{B}_{h}^{n}, s_{h})=0. 
\label{equi-picardj6}
\end{dmath}
\end{dgroup}
\end{algorithm}

The following basic properties of Algorithm \ref{alg:picardj} can be
also established similarly.

\begin{theorem}
Any solution of Algorithm \ref{alg:picardj} satisfies 
\begin{enumerate}
\item
Gauss's law of magnetic field in the strong sense:
$$
\nabla\cdot \bm{B}_{h}^{n}=0.
$$
\item
the Lagrange multiplier  $r_{h}^{n}=0$, hence \eqref{equi-picardj2} reduces to 
$$
\nabla\times (\bm{j}_{h}^{n}-\bm{\sigma}_{h}^{n})=\bm{0}.
$$
\item
the energy estimates:
 $$
R_{e}^{-1}\|\nabla\bm{u}^{n}_{h}\|^{2}+S\|\bm{j}_{h}^{n}\|^{2}=\langle \bm{f}, \bm{u}^{n}_{h}\rangle,
$$
and
\begin{align}\label{energy-picard}
\frac{1}{2R_{e}}\|\nabla\bm{u}^{n}_{h}\|^{2}+{S}\|\bm{j}^{n}_{h}\|^{2}\leq \frac{R_{e}}{2}\|\bm{f}\|_{-1}^{2}.
\end{align}
\end{enumerate}
\end{theorem}

We also recast Algorithm \ref{alg:picardj} into an abstract form of Brezzi theory for the convenience of analysis. We will use ${\bm{\xi}}_{h}$, ${\bm{\eta}}_{h}$  to denote $(\bm{u}_{h}, \bm{j}_{h}, \bm{\sigma}_{h}, \bm{B}_{h})$ and $(\bm{v}_{h}, \bm{k}_{h}, \bm{\tau}_{h}, \bm{C}_{h})$ respectively, and use ${\bm{\xi}}_{h}^{-}$ to denote $(\bm{u}^{-}_{h}, \bm{j}_{h}^{-}, \bm{\sigma}^{-}_{h}, \bm{B}^{-}_{h})$ which is the solution of last iteration step (or initial guess).
We assume $\bm{u}_{h}^{-}$ and $\bm{B}_{h}^{-}$
are given as known functions.
For the initial guess, we assume $\|\bm{u}_{h}^{0}\|_{1}$, $\|\bm{B}_{h}^{0}\|$ and $\|\bm{j}_{h}^{0}\|=\|\nabla_{h}\times\bm{B}_{h}^{0}\|$ are bounded.
From the energy estimates, we know $\|\bm{u}_{h}^{-}\|_{1}$, $\|\bm{B}_{h}^{-}\|$ and $\|\nabla_{h}\times\bm{B}_{h}^{-}\|$ are bounded uniformly with the iteration step.

Define
\begin{align*}
\bm{a}\left ({\bm{\xi}}_{h}, {\bm{\eta}}_{h}\right )=& R_{e}^{-1}(\nabla \bm{u}_{h}, \nabla \bm{v}_{h})
+L(\bm{u}^{-}_{h}; \bm{u}_{h}, \bm{v}_{h})+S(  \bm{j}_{h}, \bm{v}_{h}\times \bm{B}^{-}_{h})\\
&+ SR_{m}^{-1}(\nabla\times \bm{j}_{h},  \bm{C}_{h})-SR_{m}^{-1}(\nabla\times \bm{\sigma}_{h},   \bm{C}_{h})+SR_{m}^{-1}(\bm{\sigma}_{h}, \bm{\tau}_{h})\\
&-SR_{m}^{-1}(\bm{u}_{h}\times \bm{B}^{-}_{h}, \bm{\tau}_{h})+{S}(\bm{j}_{h}, \bm{k}_{h})-SR_{m}^{-1}(\bm{B}_{h}, \nabla\times \bm{k}_{h}).
\end{align*}
The variational form with general right hand sides can be written as:
\begin{problem}\label{prob:picardj}
Given $\bm{\xi}_{h}^{-}\in \bm{X}_{h}$, ${\bm{\theta}}=(\bm{f}, \bm{l},  \bm{g}, \bm{h})\in \bm{X}_{h}^{\ast}$, ${\bm{\psi}}=(m, z)\in \bm{Y}_{h}^{\ast} $. Find $(\bm{u}_{h}, \bm{j}_{h}, \bm{\sigma}_{h}, \bm{B}_{h}, p_{h}, r_{h})\in \bm{X}_{h}\times \bm{Y}_{h}$, such that
\begin{align}
\label{picardj1}
{\bm{a}}({\bm{\xi}}_{h}, {\bm{\eta}}_{h})+ \bm{b}({\bm{\eta}}_{h}, \bm{x}_{h})&=\langle {\bm{\theta}}, {\bm{\eta}}_{h} \rangle ,\quad \forall~ {\bm{\eta}}_{h} \in \bm{X}_{h}, \\\label{picardj2}
\bm{b}({\bm{\xi}}_{h},\bm{y}_{h})&=\langle {\bm{\psi}}, \bm{y}_{h} \rangle, \quad\forall \bm{y}_{h}\in \bm{Y}_{h}.
\end{align}
Here $\langle {\bm{\theta}}, {\bm{\eta}}_{h} \rangle:=\langle \bm{f}, \bm{v}_{h}\rangle+\langle \bm{l}, \bm{k}_{h}\rangle +\langle \bm{g}, \bm{\tau}_{h}\rangle+\langle \bm{h}, \bm{C}_{h}\rangle$, and $\langle {\bm{\psi}}, {\bm{y}}_{h} \rangle:=\langle m, q_{h}\rangle+\langle z, s_{h}\rangle$. 
\end{problem}

Problem \ref{prob:picardj} is equivalent to Algorithm \ref{alg:picardj} when $\bm{u}^{-}=\bm{u}^{n-1}$, $\bm{B}^{-}=\bm{B}^{n-1}$ and $\bm{l},  \bm{g}, \bm{h}, m, z=0$. 

We give the main theorem of well-posedness of the Picard iteration scheme:
\begin{theorem} \label{thm:well-posed-picard}
  {\bf (Well-posedness of Picard iterations)}\\
  There exists unique $(\bm{u}_{h}, \bm{j}_{h},
  \bm{\sigma}_{h}, \bm{B}_{h}, p_{h}, r_{h}) $ solving
  Problem \ref{prob:picardj}, and the solution satisfies:
$$
\|(\bm{u}_{h},\bm{j}_{h}, \bm{\sigma}_{h}, \bm{B}_{h})\|_{\bm{X}}^{2}+\|(p_{h}, r_{h})\|_{\bm{Y}}^{2} \leq C (\|(\bm{f}, \bm{l}, \bm{g}, \bm{h})\|_{\bm{X}^{\ast}}^{2}+\|(m,z)\|_{\bm{Y}^{\ast}}^{2}),
$$
where $C$ only depends on the domain,  $\|\bm{u}_{h}^{-}\|_{1}$ and $\|\bm{B}_{h}^{-}\|_{d}$.
\end{theorem}

\begin{remark}
If $\bm{u}_{h}^{-}$ and $\bm{B}_{h}^{-}$ are from the iterative scheme Algorithm \ref{alg:picardj},  $\|\bm{u}_{h}^{-}\|_{1}$ and $\|\bm{B}_{h}^{-}\|_{d}$ are uniformly bounded by known data, from the energy estimate \eqref{energy-picard}.
\end{remark}

Next we focus on the proof of this theorem. Similar to the nonlinear problem, we first formally eliminate the variable $\bm{j}_{h}$ by $\nabla_{h}\times \bm{B}_{h}$, and formally eliminate $\bm{\sigma}_{h}$ to get a system with $\bm{u}_{h}$, $\bm{B}_{h}$ and $p_{h}$, $s_{h}$ as the variables (Problem \ref{prob:picardj-equi} below). 
Boundedness and  inf-sup condition of the bilinear form $\bm{b}(\cdot,\cdot)$ are also similar to the nonlinear  problem. Finally, we use the coercivity of the bilinear form $\tilde{\bm{a}}(\tilde{\bm{\xi}}_{h}^{-};\cdot,\cdot)$  on ${\tilde{\bm{X}}_{h}^{0}}$ to get the well-posedness of the Picard iterations.

\begin{problem} \label{prob:picardj-equi}
Given $\tilde{\bm{\xi}}_{h}^{-}\in \tilde{\bm{X}}_{h}$ and $\tilde{\bm{\theta}}=(\tilde{\bm{f}}, \tilde{ \bm{h}})\in \tilde{\bm{X}}_{h}^{\ast}$, $\tilde{\bm{\psi}}=(m,z)\in \bm{Y}_{h}^{\ast}$,   find $\tilde{\bm{\xi}}_{h}\in \tilde{\bm{X}}_{h}$, $\bm{x}_{h}\in \bm{Y}_{h}$, such that
\begin{align}\label{eqn-equivalent-1}
\tilde{\bm{a}}(\tilde{\bm{\xi}}_{h}^{-}; \tilde{\bm{\xi}}_{h}, \tilde{\bm{\eta}}_{h})+ \tilde{\bm{b}}(\tilde{\bm{\eta}}_{h}, \bm{x}_{h})&=\langle \tilde{\bm{\theta}}, \tilde{\bm{\eta}}_{h} \rangle ,\quad \forall~ \tilde{\bm{\eta}}_{h} \in \tilde{\bm{X}}_{h}, \\\label{eqn-equivalent-2}
\tilde{\bm{b}}(\tilde{\bm{\xi}}_{h},\bm{y}_{h})&=\langle \tilde{\bm{\psi}}, \bm{y}_{h} \rangle, \quad\forall \bm{y}_{h}\in \bm{Y}_{h}.
\end{align}
where $\langle \tilde{\bm{f}}, \bm{v}_{h} \rangle :=\langle {\bm{f}}, \bm{v}_{h}\rangle -\langle \bm{l}, \mathbb{P}(\bm{v}_{h}\times \bm{B}_{h}^{-})\rangle$, $\langle \tilde{\bm{h}}, \bm{C}_{h}\rangle:= \langle{\bm{h}}, \bm{C}_{h} \rangle-R_{m}^{-1}\langle \bm{l}, \nabla_{h}\times \bm{C}_{h}\rangle +\langle \bm{g}, \nabla_{h}\times \bm{C}_{h}\rangle$.
\end{problem}

In what follows we use $\|\cdot\|_{c\ast}$ to denote the dual norm of $H^{h}_{0}(\mathrm{curl}, \Omega)$ (with norm $\|\cdot\|_{c}$):
$$
\|\bm{l}\|_{c\ast}:= \sup_{\bm{F}_{h}\in H_{0}^{h}(\mathrm{curl}, \Omega)}  \frac{\langle \bm{l}, \bm{F}_{h}\rangle}{\|\bm{F}_{h}\|_{c}}.
$$
  To see $\tilde{\bm{f}}$ and $\tilde{\bm{h}}$ are bounded linear operators, we note the basic estimates:
\begin{align*}
\langle \bm{l}, \mathbb{P}(\bm{v}_{h}\times \bm{B}_{h}^{-}) \rangle &\leq \|\bm{l}\|_{c\ast}\|\mathbb{P}(\bm{v}_{h}\times \bm{B}_{h}^{-})\|_{c}\\&
\leq \|\bm{l}\|_{c\ast}\|\bm{v}_{h}\times \bm{B}_{h}^{-}\|\\&
\lesssim \|\bm{l}\|_{c\ast}\|\bm{v}_{h}\|_{1}\|\nabla_{h}\times \bm{B}_{h}^{-}\|,
\end{align*}
and
$$
\langle \bm{l}, \nabla_{h}\times \bm{C}_{h}\rangle \leq \|\bm{l}\|_{c\ast}\|\nabla_{h}\times \bm{C}_{h}\|_{c}\leq \|\bm{l}\|_{c\ast}\| \bm{C}_{h}\|_{d},
$$
$$
\langle \bm{g}, \nabla_{h}\times \bm{C}_{h}\rangle \leq \|\bm{g}\|_{c\ast}\|\nabla_{h}\times \bm{C}_{h}\|_{c}\leq \|\bm{g}\|_{c\ast}\| \bm{C}_{h}\|_{d}.
$$

In the following discussion, we will use the Riesz representation $\bm{l}_{0}, \bm{g}_{0}\in H_{0}^{h}(\mathrm{curl}, \Omega)$ of $\bm{l}, \bm{g}\in H_{0}^{h}(\mathrm{curl}, \Omega)^{\ast}$ which are defined by
$$
(\bm{g}_{0}, \bm{\tau}_{h}):=\langle \bm{g}, \bm{\tau}_{h} \rangle, \quad \forall \bm{\tau}_{h}\in H_{0}^{h}(\mathrm{curl}, \Omega),
$$
and
$$
(\bm{l}_{0}, \bm{k}_{h}):=\langle \bm{l}, \bm{k}_{h} \rangle, \quad \forall \bm{k}_{h}\in H_{0}^{h}(\mathrm{curl}, \Omega).
$$
Note that $\|\bm{g}_{0}\|_{c}=\|\bm{g}\|_{c\ast}$ and $\|\bm{l}_{0}\|_{c}=\|\bm{l}\|_{c\ast}$.

For the relation between Problem \ref{prob:picardj} and Problem \ref{prob:picardj-equi}, we have:
\begin{lemma}\label{lem:equivalence}
If $(\bm{u}_{h}, \bm{B}_{h}, p_{h}, r_{h})$ solves Problem \ref{prob:picardj-equi} and 
$$
\|\bm{u}_{h}\|_{1}^{2}+\|\bm{B}_{h}\|_{d}^{2}+\|p_{h}\|^{2}+\|r_{h}\|^{2}\leq c_{1}( \|\tilde{\bm{f}}\|_{-1}^{2}+\|\tilde{\bm{h}}\|_{H^{h}(\mathrm{div})^{\ast}}^{2}+\|(m,z)\|_{\bm{Y}^{\ast}}^{2}),
$$
then
\begin{align*}
(\bm{u}_{h}, \bm{j}_{h}, \bm{\sigma}_{h}, \bm{B}_{h}, p_{h}, r_{h}):=(\bm{u}_{h}, &R_{m}^{-1}\nabla_{h}\times \bm{B}_{h}+S^{-1}\bm{l}_{0}, \\&\mathbb{P}(\bm{u}_{h}\times \bm{B}_{h}^{-})
+S^{-1}{R_m}\bm{g}_{0}, \bm{B}_{h}, p_{h}, r_{h})\in \bm{X}_{h}\times \bm{Y}_{h}
\end{align*}
 solves Problem \ref{prob:picardj}, and there exists a positive constant $c_{2}$, depending on $c_{1}$ and $\|\bm{B}_{h}^{-}\|_{d}$ such that
\begin{align}\label{norm-control2}
\|(\bm{u}_{h}, \bm{j}_{h}, \bm{\sigma}_{h}, \bm{B}_{h})\|^{2}_{\bm{X}}+\|(p_{h}, 
r_{h})\|^{2}_{\bm{Y}}   \leq c_{2}\left ( \|(\bm{f}, \bm{l}, \bm{g}, \bm{h})\|_{\bm{X}^{\ast}}^{2}+\|(m,z)\|_{\bm{Y}^{\ast}}^{2}\right ).
\end{align}

On the other hand, if $(\bm{u}_{h}, \bm{j}_{h}, \bm{\sigma}_{h}, \bm{B}_{h}, p_{h}, r_{h})$ solves Problem \ref{prob:picardj}, then $(\bm{u}_{h}, \bm{B}_{h}, p_{h}, r_{h})$ solves Problem \ref{prob:picardj-equi}.

\end{lemma}

\begin{proof}
In Problem \ref{prob:picardj},  we take $\bm{v}_{h},  \bm{k}_{h}, \bm{C}_{h},  q_{h},s_{h}=0$ in \eqref{picardj1} to see   
\begin{align}\label{def:sigma}
\bm{\sigma}_{h}=\mathbb{P}(\bm{u}_{h}\times \bm{B}_{h}^{-})+S^{-1}{R_m}\bm{g}_{0},
\end{align}
and
take $\bm{v}_{h},  \bm{\tau}_{h},\bm{C}_{h},  q_{h},s_{h}=0$ in \eqref{picardj1} to get
\begin{align}\label{def:j}
\bm{j}_{h}=R_m^{-1}\nabla_{h}\times \bm{B}_{h}+S^{-1}\bm{l}_{0}.
\end{align}

If $(\bm{u}_{h}, \bm{B}_{h}, p_{h}, r_{h})$ solves Problem \ref{prob:picardj-equi}, and
$$
\|\bm{u}_{h}\|_{1}^{2}+\|\bm{B}_{h}\|_{d}^{2}+\|p_{h}\|^{2}+\|r_{h}\|^{2}\leq c_{1}\left (\|\tilde{\bm{f}}\|_{-1}^{2}+\|\tilde{\bm{h}}\|_{H^{h}(\mathrm{div})^{\ast}}^{2}+\|(m,z)\|_{\bm{Y}^{\ast}}^{2}\right ),
$$
it is easy to see from \eqref{def:sigma} and \eqref{def:j} that $(\bm{u}_{h}, R_{m}^{-1}\nabla_{h}\times \bm{B}_{h}+S^{-1}\bm{l}_{0}, \mathbb{P}(\bm{u}_{h}\times \bm{B}_{h}^{-})+S^{-1}{R_m}\bm{g}_{0}, \bm{B}_{h}, p_{h}, r_{h})$ solves Problem  \ref{prob:picardj}, and
\begin{align*}
R_{m}^{-2}\|\nabla_{h}\times \bm{B}_{h}\|_{c}^{2}+\|\bm{B}_{h}\|_{d}^{2}&=\|\bm{B}_{h}\|^{2}+\|\nabla\cdot \bm{B}_{h}\|^{2}+(1+R_{m}^{-2})\|\nabla_{h}\times \bm{B}_{h}\|^{2}\\
&\lesssim \|\bm{B}_{h}\|_{d}^{2},
\end{align*}
\begin{align*}
\|\mathbb{P}(\bm{u}_{h}\times \bm{B}_{h}^{-})\| \leq \|\bm{u}_{h}\times \bm{B}_{h}^{-}\|\lesssim \|\bm{u}_{h}\|_{1}\|\nabla_{h}\times\bm{B}_{h}^{-}\|.
\end{align*}
This implies \eqref{norm-control2}.


On the other hand, solution of Problem \ref{prob:picardj} also solves Problem \ref{prob:picardj-equi} by substituting \eqref{def:sigma} and \eqref{def:j} into \eqref{picardj1}.

\end{proof}

Once the well-posedness of Problem \ref{prob:picardj-equi} is established,  the first part of Lemma \ref{lem:equivalence} will imply existence and stability of the original Problem \ref{prob:picardj}, and the second part will imply the uniqueness.
Hence it suffices to prove well-posedness of Problem \ref{prob:picardj-equi} under the norm $\|\cdot\|_{A}$ (\eqref{norm}).

Similar to the nonlinear case, we have
\begin{lemma} {\bf (Boundedness)  }  \label{lem:boundedness}
 $\tilde{\bm{a}}(\tilde{\bm{\xi}}_{h}^{-};\cdot,\cdot)$ is a bounded bilinear form on ${\tilde{\bm{X}}_{h}^{0}}$ with respect to $\|\cdot\|_{A}$ (\eqref{norm})\end{lemma}

We note that the bound depends on the domain and $\|\bm{u}_{h}^{-}\|_{0,3}$, $\|\nabla_{h}\times\bm{B}_{h}^{-}\|$. By the energy estimates, we know these terms are bounded by known data.

The boundedness and inf-sup condition of $\bm{b}({\cdot,\cdot})$ are the same as the nonlinear problem (Lemma \ref{lem:nonlinear-boundedness2}, Lemma \ref{lem:inf-supb}).

Next we show the coercivity of $\tilde{\bm{a}}(\tilde{\bm{\xi}}_{h}^{-};\cdot,\cdot)$  on ${\tilde{\bm{X}}_{h}^{0}}$:
\begin{lemma}\label{lem:coercivity-a}
There exists a positive constant $\alpha$ such that
$$
\tilde{\bm{a}}(\tilde{\bm{\xi}}_{h}^{-}; \tilde{\bm{\xi}}_{h}, \tilde{\bm{\xi}}_{h})\geq \alpha (\|\bm{u}_{h}\|_{1}^{2}+  \| \bm{B}_{h}\|_{d}^{2}),\quad \forall \tilde{\bm{\xi}}_{h}\in {\tilde{\bm{X}}_{h}^{0}}.
$$
\end{lemma}

\begin{proof}
Taking $\bm{v}_{h}=\bm{u}_{h}$, $\bm{C}_{h}=\bm{B}_{h}$,
$$
\tilde{\bm{a}}(\tilde{\bm{\xi}}_{h}^{-}; \tilde{\bm{\xi}}_{h}, \tilde{\bm{\xi}}_{h})=R_{e}^{-1}\| \nabla \bm{u}_{h}\|^{2}+SR_{m}^{-2} \| \nabla_{h}\times \bm{B}_{h}\|^{2}.
$$

 From Poincar\'{e}'s inequality (Lemma \ref{lem:poincare}) and $\nabla\cdot \bm{B}_{h}=0$ on ${\tilde{\bm{X}}_{h}^{0}}$:
$$
\|\bm{B}_{h}\|\lesssim \| \nabla_{h}\times \bm{B}_{h}\|.
$$

Hence 
$$
\|\bm{B}_{h}\|_{d}\lesssim \| \nabla_{h}\times \bm{B}_{h}\|,
$$
and there exists a positive constant $\alpha$ which only depends on the domain and $R_{e}$, $R_{m}$, $S$ such that
$$
\tilde{\bm{a}}(\tilde{\bm{\xi}}_{h}^{-}; \tilde{\bm{\xi}}_{h}, \tilde{\bm{\xi}}_{h})\geq \alpha (\| \bm{u}_{h}\|_{1}^{2}+  \| \bm{B}_{h}\|_{d}^{2}).
$$
\end{proof}

From Lemma \ref{lem:boundedness}, Lemma \ref{lem:inf-supb} and Lemma
\ref{lem:coercivity-a}, we have proved the well-posedness of Problem
\ref{prob:picardj-equi}. From Lemma \ref{lem:equivalence}, this
shows the well-posedness of Problem \ref{prob:picardj}, and hence
Algorithm \ref{alg:picardj} as a special case.

\section{Convergence analysis} \label{sec:convergence}

\subsection{Convergence of Picard iterations}

There is a general argument to prove the convergence of Picard iterations under the condition of small data, which guarantees the uniqueness of the nonlinear scheme (c.f. Girault and Raviart \cite{Girault.V;Raviart.P.1986a} Chapter IV, Remark 1.3;  Gunzburger et al. \cite{Gunzburger.M;Meir.A;Peterson.J.1991a} Proposition 7.1). Since we have established the boundedness and coercivity of the nonlinear variational form,  the convergence of  Picard iteration scheme proposed in this paper can be analyzed in the same way, and a comparable result holds. However we note that in the condition obtained in this way, the coupling number $S$ cannot be arbitrarily small, which seems to be contrary to the physical intuition. For example, in Gunzburger et al. \cite{Gunzburger.M;Meir.A;Peterson.J.1991a}, when we assume that the boundary data is zero, the criterion ((4.26) of \cite{Gunzburger.M;Meir.A;Peterson.J.1991a}) is reduced to
\begin{align}\label{picard-condition0}
\|\bm{f}\|_{-1}<\frac{S}{\sqrt{2}\gamma_{3}}\frac{ \left (  \min\left ( \frac{k_{1}}{SR_{e}}, \frac{k_{2}}{R_{m}^{2}}\right ) \right )^{2}}{\max\left ( \frac{1}{S}, \frac{\sqrt{2}}{R_{m}}  \right )}.
\end{align}
Here we have used the notation in \eqref{mhd}, with a correspondence to the original notation in \cite{Gunzburger.M;Meir.A;Peterson.J.1991a}: $N=S$, $M=\sqrt{SR_{e}}$, $F=S^{-1}\bm{f}$, where $F$ is the right hand side in  \cite{Gunzburger.M;Meir.A;Peterson.J.1991a}. Furthermore, here $\gamma_{3}$, $k_{1}$ and $k_{2}$ are positive constants in the Sobolev imbedding and the Poincar\'{e}'s inequality of velocity and magnetic fields. Now it is easy to see that in \eqref{picard-condition0}, $S$ cannot be arbitrarily small for fixed $R_{e}$, $R_{m}$ and $\bm{f}\neq \bm{0}$.
The condition (2.16) in Sch\"{o}tzau \cite{Schotzau.D.2004a} is similar.

 Therefore in this section, we use a different approach and directly prove the convergence of the Picard iterations by contraction. As a result, we will see that the small data condition (\eqref{small-data1}  below) will only contain $R_{e}$ and $R_{m}$, but not $S$.  The (discrete) energy law is crucial in the argument below as an a priori estimate.

A similar argument also holds on  the continuous level with minor modifications. We omit the subscript ``$h$'' in this section.
\begin{theorem}
The Picard iteration scheme (Algorithm \ref{alg:picardj}) converges when
\begin{align}\label{small-data1}
\|\bm{f}\|_{-1}\leq   \left ( 2C_{1}^{4}R_e^{4}+ 4C_{2}^{2}R_e^{2} R_m^{2} \right )^{-\frac{1}{2}},
\end{align}
where $C_{1}$ and $C_{2}$, depending only on the domain, are positive constants related to the Sobolev imbedding and regularity estimates of $H^{h}(\mathrm{div})$ functions given in \eqref{def:C1} and \eqref{def:C2}.
\end{theorem}
The above conditions are satisfied when the data $\|\bm{f}\|_{-1}$ is small relative to ${R_e^{-1}}$ and $R_{m}^{-1}$. 
\begin{proof}
By the standard theory of mixed methods, it suffices to consider the convergence in $\bm{X}_{h}^{0}:=\{ {\bm{\eta}}_{h}\in {\bm{X}}_{h}: {\bm{b}}({\bm{\eta}}_{h}, \bm{y}_{h})=0, ~\forall \bm{y}_{h}\in \bm{Y}_{h}   \}.$

The equation of the $n$-th step  can be written as
\begin{align}\label{errorn1}
L(\bm{u}^{n-1}; \bm{u}^{n}, \bm{v})+R_{e}^{-1}(\nabla \bm{u}^{n}, \nabla \bm{v})-S(\bm{j}^{n}\times \bm{B}^{n-1}, \bm{v})&=\langle \bm{f}, \bm{v}\rangle,\\\label{errorn2}
-(\bm{u}^{n}\times \bm{B}^{n-1}, \nabla_{h}\times \bm{C})+R_{m}^{-1}(\nabla_{h}\times \bm{B}^{n}, \nabla_{h}\times \bm{C})&=0.
\end{align}
The $(n-1)$-th step  is similarly written as
\begin{align}\label{errorn-11}
L(\bm{u}^{n-2}; \bm{u}^{n-1}, \bm{v})+R_{e}^{-1}(\nabla \bm{u}^{n-1}, \nabla \bm{v})-S(\bm{j}^{n-1}\times \bm{B}^{n-2}, \bm{v})&=\langle \bm{f}, \bm{v}\rangle,\\
-(\bm{u}^{n-1}\times \bm{B}^{n-2}, \nabla_{h}\times \bm{C})+R_{m}^{-1}(\nabla_{h}\times \bm{B}^{n-1}, \nabla_{h}\times \bm{C})&=0.\label{errorn-12}
\end{align}
Define the errors 
$$
e_{u}^{n}:=\bm{u}^{n}-\bm{u}^{n-1}, \quad e_{B}^{n}:=\bm{B}^{n}-\bm{B}^{n-1}, \quad e_{j}^{n}:=\bm{j}^{n}-\bm{j}^{n-1}.
$$
From the equation $\bm{j}^{n}=R_{m}^{-1}\nabla_{h}\times \bm{B}^{n}$, we have $e_{j}^{n}=R_{m}^{-1}\nabla_{h}\times e_{B}^{n}$.

Subtracting \eqref{errorn-11}-\eqref{errorn-12} from the $n$-step equation \eqref{errorn1}-\eqref{errorn2}, we get the error equation:
\begin{align*}
  \frac{1}{2}&\left ( (\bm{u}^{n-1}\cdot \nabla) e_{u}^{n}, \bm{v}
  \right )+\frac{1}{2}\left ( (e_{u}^{n-1}\cdot\nabla)\bm{u}^{n-1},
    \bm{v} \right )-\frac{1}{2}\left (
    (\bm{u}^{n-1}\cdot\nabla)\bm{v}, e_{u}^{n} \right )\\&
  -\frac{1}{2}\left ( (e_{u}^{n-1}\cdot\nabla)\bm{v}, \bm{u}^{n-1} \right )+\frac{1}{R_{e}}(\nabla e_{u}^{n}, \nabla \bm{v})+S\left (\bm{B}^{n-1}\times e_{j}^{n}, \bm{v} \right )+S\left (e_{B}^{n-1}\times \bm{j}^{n-1}, \bm{v}\right )=0,\\
  -&( e_{u}^{n}\times \bm{B}^{n-1}, \nabla_{h}\times \bm{C}
  )-(\bm{u}^{n-1}\times e_{B}^{n-1}, \nabla_{h}\times
  \bm{C})+R_{m}^{-1}(\nabla_{h}\times e_{B}^{n}, \nabla_{h}\times
  \bm{C})=0.
\end{align*}
Multiplying the second equation by $SR_{m}^{-1}$,  adding the above two equations and taking $\bm{v}=e_{u}^{n}$,
$\bm{C}=e_{B}^{n}$ yield
\begin{align}\label{eqn:error1}
&\quad \frac{1}{2}\left( (e_{u}^{n-1}\cdot\nabla)\bm{u}^{n-1}, e_{u}^{n}
  \right )-\frac{1}{2}\left ( (e_{u}^{n-1}\cdot \nabla)e_{u}^{n} ,
    \bm{u}^{n-1} \right)+R_{e}^{-1}(\nabla e_{u}^{n}, \nabla
  e_{u}^{n})+S(e_{B}^{n-1}\times \bm{j}^{n-1},
  e_{u}^{n})
  \\\nonumber &
   -SR_{m}^{-1}\left (\bm{u}^{n-1}\times e_{B}^{n-1},
  \nabla_{h}\times e_{B}^{n}\right )  +SR_{m}^{-2}(\nabla_{h}\times e_{B}^{n},
  \nabla_{h}\times e_{B}^{n})=0.
\end{align}
From the energy estimates \eqref{energy-picard}, we know 
$$
\|\nabla\bm{u}^{n}\|\leq R_{e}\|\bm{f}\|_{-1},
$$
and 
$$
\|\bm{j}^{n}\|\leq \left (\frac{R_{e}}{2S}\right )^{\frac{1}{2}}\|\bm{f}\|_{-1},
$$
which hold for all $n>0$.

  Then we have the estimates for the nonlinear terms:
\begin{align*}
\left |\frac{1}{2} \left(   (e_{u}^{n-1}\cdot\nabla)\bm{u}^{n-1}, e_{u}^{n}  \right ) \right |&\leq \frac{1}{2}\|e_{u}^{n-1}\|_{0,3}\|\nabla \bm{u}^{n-1}\|\|e_{u}^{n}\|_{0,6}\\&
\leq \frac{1}{2}C_{1}^{2}R_{e}\|\bm{f}\|_{-1}  \|\nabla e_{u}^{n-1}\| \|\nabla e_{u}^{n}\|\\&
\leq \frac{1}{8R_{e}}\|\nabla e_{u}^{n-1}\|^{2}+\frac{1}{2}C_{1}^{4}R_{e}^{3}\|\bm{f}\|_{-1}^{2}\|\nabla e_{u}^{n}\|^{2},
\end{align*}
\begin{align*}
\left |\frac{1}{2} \left(   (e_{u}^{n-1}\cdot\nabla)e_{u}^{n}, \bm{u}^{n-1}  \right ) \right |&\leq \frac{1}{2}\|\bm{u}^{n-1}\|_{0,6}\|\nabla e_{u}^{n}\|\|e_{u}^{n-1}\|_{0,3}\\&
\leq \frac{1}{2}C_{1}^{2}R_{e}\|\bm{f}\|_{-1}  \|\nabla e_{u}^{n-1}\| \|\nabla e_{u}^{n}\|\\&
\leq \frac{1}{8R_{e}}\|\nabla e_{u}^{n-1}\|^{2}+\frac{1}{2}C_{1}^{4}R_{e}^{3}\|\bm{f}\|_{-1}^{2}\|\nabla e_{u}^{n}\|^{2},
\end{align*}
\begin{align*}
\left | S(e_{B}^{n-1}\times  \bm{j}^{n-1}, e_{u}^{n} ) \right |  &\leq SC_{2}\|\nabla_{h}\times e_{B}^{n-1}\|\|\bm{j}^{n-1}\|\|\nabla e_{u}^{n}\|\\&
\leq SC_{2}R_{m}\|\bm{j}^{n-1}\|\|e_{j}^{n-1}\|\|\nabla e_{u}^{n}\|\\&
\leq SC_{2}R_{m}\left ( \frac{R_{e}}{2S} \right )^{\frac{1}{2}}\|\bm{f}\|_{-1}\|e_{j}^{n-1}\|\|\nabla e_{u}^{n}\|\\&
\leq \frac{1}{8}S\|e_{j}^{n-1}\|^{2}+2R_{e}C_{2}^{2}R_{m}^{2}\|\bm{f}\|_{-1}^{2}\|\nabla e_{u}^{n}\|^{2},
\end{align*}
and
\begin{align*}
\left | SR_{m}^{-1}(\bm{u}^{n-1}\times e_{B}^{n-1}, \nabla_{h}\times e_{B}^{n}    )  \right |&\leq SR_{m}^{-1}C_{2}\|\nabla_{h}\times e_{B}^{n-1}\|\|\nabla\bm{u}^{n-1}\|\|\nabla_{h}\times e_{B}^{n}\|\\&
\leq SC_{2}R_{e}R_{m}\|\bm{f}\|_{-1}\|e_{j}^{n-1}\|\|e_{j}^{n}\|\\&
\leq \frac{1}{8}S\|e_{j}^{n-1}\|^{2}+2SR_{m}^{2}C_{2}^{2}R_{e}^{2}\|\bm{f}\|_{-1}^{2}\|e_{j}^{n}\|^{2}.
\end{align*}

Combining the above estimates with \eqref{eqn:error1}, we have
\begin{align*}
&\left(\frac{1}{R_{e}}-C_{1}^{4}R_{e}^{3}\|\bm{f}\|_{-1}^{2}-2R_{e} C_{2}^{2}R_{m}^{2} \|\bm{f}\|_{-1}^{2}   \right)\|\nabla e_{u}^{n}\|^{2}\\&
+\left (  {S}-2R_{m}^{2}SC_{2}^{2}R_{e}^{2}\|\bm{f}\|_{-1}^{2}      \right)\|e_{j}^{n}\|^{2}\leq \frac{1}{4R_{e}}\|\nabla e_{u}^{n-1}\|^{2}+\frac{1}{4}S\|e_{j}^{n-1}\|^{2}.
\end{align*}
We define the energy functional to be
$$
\mathcal{E}^{n}:=\frac{1}{2R_{e}}\|\nabla e_{u}^{n}\|^{2}+\frac{1}{2}S\|e_{j}^{n}\|^{2}.
$$
Therefore when 
$$
\frac{1}{2R_{e}}\geq C_{1}^{4}R_{e}^{3}\|\bm{f}\|_{-1}^{2}+2C_{2}^{2}R_{e} R_{m}^{2} \|\bm{f}\|_{-1}^{2},
$$
and
$$
\frac{1}{2}S\geq 2R_{m}^{2} S C_{2}^{2}R_{e}^{2}\|\bm{f}\|_{-1}^{2},
$$
i.e. when \eqref{small-data1} holds,
we have 
$$
\mathcal{E}^{n}\leq \frac{1}{2}\mathcal{E}^{n-1}.
$$
This implies that $(\bm{u}^{n}, \bm{B}^{n})$ converges to some $(\bm{u}, \bm{B})$ in the norm defined by 
$$
R_{e}^{-1}\|\nabla \bm{u}^{n}\|^{2}+SR_{m}^{-2}\|\nabla_{h}\times \bm{B}^{n}\|^{2}.
$$
Combined with the continuity of the trilinear form, we can take the limit  and $(\bm{u}, \bm{B})$ is a solution of the nonlinear Problem \ref{prob:nonlinear}.

From the inf-sup condition of the velocity-pressure pair, we also have the convergence of the pressure $p^{n}$.
\end{proof}


\subsection{Convergence of the finite element method}

We prove the convergence of the nonlinear finite element scheme.  In the discussions below, we  deal with the reduced form of the finite element scheme with variables $(\bm{u}_{h},  \bm{B}_{h}, p_{h}, r_{h})$ (Problem \ref{prob:nonlinear-equi}), then recover $\bm{j}_{h}$ and $\bm{\sigma}_{h}$ from these variables. 

As a routine approach for mixed methods, the proof below consists of several steps. We first subtract the finite element solution from the true solution to obtain certain orthogonality (\eqref{orthogonality}). Then we insert an arbitrary discrete function to the orthogonality equation to get \eqref{eqn:formal-system}. Combining with triangular inequalities, numerical errors can be bounded by the difference of the true solution and the discrete functions inserted above. Such an estimate is usually called quasi-orthogonality (Theorem \ref{thm:quasi-orthogonality}).   Then the final estimate (\eqref{error}) follows from polynomial approximation results.

The analysis below also contains some new features compared with conventional error estimates for mixed methods. The finite element scheme involves the discrete adjoint operator $\nabla_{h}\times $, which can only be defined for finite element functions. Therefore it is no wonder that the consistency error $\|\nabla\times \bm{B}-\nabla_{h}\times \bm{B}_{I}\|$ will come into our analysis. Moreover, in the analysis for the nonlinear problem, we will frequently use the key technical results established in Section \S \ref{sec:estimates} to provide the {\it a priori} estimate for both numerical and true solutions. Combining these key estimates and small source assumptions, which are common for nonlinear problems, we obtain the desired results. 

We begin detailed analysis by discovering the  orthogonality.
Subtracting the true solution of \eqref{curl-formulation} from the variational form \eqref{equi-fem}, we have  for any $(\bm{v}_{h}, \bm{C}_{h})\in \tilde{\bm{X}}_{h}$, $(q_{h}, s_{h})\in \bm{Y}_{h}$,
\begin{align}\label{orthogonality}\footnotesize 
\begin{cases}
&\frac{1}{2} [\left ( \left ( \bm{u}_{h}-\bm{u}\right )\cdot \nabla \bm{u}_{h}, \bm{v}_{h}\right )+\left ( \left ( \bm{u}\cdot \nabla \right )\left (\bm{u}_{h}-\bm{u}\right), \bm{v}_{h}\right ) - \left ( (\bm{u}_{h}\cdot\nabla)\bm{v}_{h}, \bm{u}_{h}- \bm{u} \right )\\
&\quad\quad  -\left (\left ( \bm{u}_{h}-\bm{u}\right )\cdot\nabla\bm{v}_{h}, \bm{u} \right ) ]+R_{e}^{-1}\left ( \nabla(\bm{u}_{h}-\bm{u}), \nabla\bm{v}_{h}\right )-\left ( p_{h}-p, \nabla\cdot \bm{v}_{h}\right )\\&\quad\quad\quad\quad\quad\quad\quad -SR_{m}^{-1}\left( ( \nabla_{h}\times \bm{B}_{h})\times \bm{B}_{h}, \bm{v}_{h} \right )
+SR_{m}^{-1}\left ( \left ( \nabla\times \bm{B}\right )\times \bm{B}, \bm{v}_{h}\right )= 0,\\ 
&-SR_{m}^{-1} \left ( \bm{u}_{h}\times \bm{B}_{h}, \nabla_{h}\times \bm{C}_{h}\right )+SR_{m}^{-1}\left (\nabla\times (\bm{u}\times \bm{B}),  \bm{C}_{h}\right )\\ 
&\quad\quad\quad\quad\quad\quad+SR_{m}^{-2}\left (\nabla\times\nabla_{h}\times \bm{B}_{h}-\nabla\times \nabla\times \bm{B},  \bm{C}_{h} \right )+(r_{h}-r, \nabla\cdot\bm{C}_{h})=0,\\  &\left  (\nabla\cdot (\bm{u}_{h}-\bm{u}), q_{h}\right )=0, \\
& \left ( \nabla\cdot (\bm{B}_{h}-\bm{B}), s_{h}\right )=0.
\end{cases}
\end{align}

We assume that $(\bm{u}_{I}, \bm{B}_{I})\in \tilde{\bm{X}}_{h}$ and $(p_{I}, r_{I})\in \bm{Y}_{h}$ are arbitrary discrete functions. Inserting $(\bm{u}_{I}, \bm{B}_{I})$, $(p_{I}, r_{I})$ into \eqref{orthogonality}, we get: for any $(\bm{v}_{h}, \bm{C}_{h})\in \tilde{\bm{X}}_{h}$, $(q_{h}, s_{h})\in \bm{Y}_{h}$, 
\begin{align*}\footnotesize
\begin{cases}
&\frac{1}{2} [\left ( \left ( \bm{u}_{h}-\bm{u}_{I}\right )\cdot \nabla \bm{u}_{h}, \bm{v}_{h}\right )+\left ( \left ( \bm{u}\cdot \nabla \right )\left (\bm{u}_{h}-\bm{u}_{I}\right), \bm{v}_{h}\right ) - \left ( (\bm{u}_{h}\cdot\nabla)\bm{v}_{h}, \bm{u}_{h}- \bm{u}_{I} \right ) \\
& \quad\quad -\left (\left ( \bm{u}_{h}-\bm{u}_{I}\right )\cdot\nabla\bm{v}_{h}, \bm{u} \right ) ]+R_{e}^{-1}\left ( \nabla(\bm{u}_{h}-\bm{u}_{I}), \nabla\bm{v}_{h}\right )-\left ( p_{h}-p_{I}, \nabla\cdot \bm{v}_{h}\right )\\
&\quad\quad -SR_{m}^{-1}\left(  \nabla_{h}\times (\bm{B}_{h}-\bm{B}_{I})\times \bm{B}_{h}, \bm{v}_{h} \right )-SR_{m}^{-1}\left ( \left ( \nabla\times \bm{B}\right )\times \left (\bm{B}_{h}-\bm{B}_{I} \right ), \bm{v}_{h}\right )\\&
\quad\quad\quad\quad= \frac{1}{2} [ (  ( \bm{u}-\bm{u}_{I})\cdot \nabla  )\bm{u}_{h}, \bm{v}_{h} )+(  ( \bm{u}\cdot \nabla  ) (\bm{u}-\bm{u}_{I}), \bm{v}_{h})  -  ( (\bm{u}_{h}\cdot\nabla)\bm{v}_{h}, \bm{u}- \bm{u}_{I} ) \\
&\quad\quad\quad\quad\quad\quad - ( ( \bm{u}-\bm{u}_{I} )\cdot\nabla\bm{v}_{h}, \bm{u}  ) ]+R_{e}^{-1}\left ( \nabla(\bm{u}-\bm{u}_{I}), \nabla\bm{v}_{h}\right )-\left ( p-p_{I}, \nabla\cdot \bm{v}_{h}\right )\\
&\quad\quad\quad\quad\quad\quad +SR_{m}^{-1}\left ( (\nabla_{h}\times \bm{B}_{I}-\nabla\times \bm{B})\times \bm{B}_{h}, \bm{v}_{h}\right )+SR_{m}^{-1}\left ( (\nabla\times \bm{B})\times (\bm{B}_{I}-\bm{B}), \bm{v}_{h}\right ),\\
&-SR_{m}^{-1} \left ( (\bm{u}_{h}-\bm{u}_{I})\times \bm{B}_{h}, \nabla_{h}\times \bm{C}_{h}\right )-SR_{m}^{-1}\left (\bm{u}\times (\bm{B}_{h}-\bm{B}_{I}), \nabla_{h}\times \bm{C}_{h}\right )\\&
\quad+SR_{m}^{-2}\left (\nabla_{h}\times (\bm{B}_{h}-\bm{B}_{I}), \nabla_{h}\times \bm{C}_{h} \right )+(r_{h}-r_{I}, \nabla\cdot\bm{C}_{h})
\\&\quad\quad=-SR_{m}^{-1} \left ( (\bm{u}-\bm{u}_{I})\times \bm{B}_{h},  \nabla_{h}\times \bm{C}_{h}\right )
+SR_{m}^{-1}\left (\bm{u}\times (\bm{B}_{I}-\bm{B}), \nabla_{h}\times \bm{C}_{h} \right )
\\&+SR_{m}^{-2}\left (\nabla\times (\nabla\times \bm{B}- \nabla_{h}\times \bm{B}_{I}),  \bm{C}_{h} \right )
+(r-r_{I}, \nabla\cdot\bm{C}_{h})-SR_{m}^{-1}(\nabla\times (\mathrm{id}-\mathbb{P})(\bm{u}\times \bm{B}), \bm{C}_{h}),\\
&\left  (\nabla\cdot (\bm{u}_{h}-\bm{u}_{I}), q_{h}\right )=\left ( \nabla\cdot (\bm{u}-\bm{u}_{I}), q_{h}\right ), \\
& \left ( \nabla\cdot (\bm{B}_{h}-\bm{B}_{I}), s_{h}\right )=\left ( \nabla\cdot (\bm{B}-\bm{B}_{I}), s_{h}\right ).
\end{cases}
\end{align*}
Here we have used the identity 
$$
(\nabla\times (\bm{u}\times \bm{B}), \bm{C}_{h})=(\nabla\times (\mathrm{id}-\mathbb{P})(\bm{u}\times \bm{B}), \bm{C}_{h})+(\bm{u}\times \bm{B}, \nabla_{h}\times \bm{C}_{h}).
$$

Adding the first two equations together, we can write the above system  as 
\begin{align}\label{eqn:formal-system}\footnotesize
\begin{cases}
&R_{e}^{-1}\left ( \nabla(\bm{u}_{h}-\bm{u}_{I}), \nabla\bm{v}_{h}\right )-\left ( p_{h}-p_{I}, \nabla\cdot \bm{v}_{h}\right )-SR_{m}^{-1}\left(  \nabla_{h}\times (\bm{B}_{h}-\bm{B}_{I})\times \bm{B}_{h}, \bm{v}_{h} \right )\\
&\quad-SR_{m}^{-1} \left ( (\bm{u}_{h}-\bm{u}_{I})\times \bm{B}_{h}, \nabla_{h}\times \bm{C}_{h}\right )+SR_{m}^{-2}\left (\nabla_{h}\times (\bm{B}_{h}-\bm{B}_{I}), \nabla_{h}\times \bm{C}_{h} \right )\\
&\quad+(r_{h}-r_{I}, \nabla\cdot\bm{C}_{h})+G(\bm{u}_{h}, \bm{B}_{h}, \bm{u}, \bm{B}; \bm{u}_{h}-\bm{u}_{I}, \bm{B}_{h}-\bm{B}_{I}; \bm{v}_{h}, \bm{C}_{h})\\&
=H(\bm{u}_{h}, \bm{B}_{h}, \bm{u}, \bm{B}; \bm{u}-\bm{u}_{I}, \bm{B}-\bm{B}_{I}, p-p_{I}, r-r_{I}; \bm{v}_{h}, \bm{C}_{h})+SR_{m}^{-1}\left ( (\nabla_{h}\times \bm{B}_{I}-\nabla\times \bm{B})\times \bm{B}_{h}, \bm{v}_{h}\right )\\
&\quad\quad\quad+SR_{m}^{-2}\left (\nabla\times (\nabla\times \bm{B}- \nabla_{h}\times \bm{B}_{I}),  \bm{C}_{h} \right )-SR_{m}^{-1}(\nabla\times (\mathrm{id}-\mathbb{P})(\bm{u}\times \bm{B}), \bm{C}_{h}),\\
&\left  (\nabla\cdot (\bm{u}_{h}-\bm{u}_{I}), q_{h}\right )=\left ( \nabla\cdot (\bm{u}-\bm{u}_{I}), q_{h}\right ), \\ 
& \left ( \nabla\cdot (\bm{B}_{h}-\bm{B}_{I}), s_{h}\right )=\left ( \nabla\cdot (\bm{B}-\bm{B}_{I}), s_{h}\right ),
\end{cases}
\end{align}
where 
\begin{align*}
G(\bm{u}_{h},& \bm{B}_{h}, \bm{u}, \bm{B}; \bm{u}_{h}-\bm{u}_{I}, \bm{B}_{h}-\bm{B}_{I}; \bm{v}_{h}, \bm{C}_{h})=\frac{1}{2} [\left ( \left ( \bm{u}_{h}-\bm{u}_{I}\right )\cdot \nabla \bm{u}_{h}, \bm{v}_{h}\right )\\&+\left ( \left ( \bm{u}\cdot \nabla \right )\left (\bm{u}_{h}-\bm{u}_{I}\right), \bm{v}_{h}\right ) - \left ( (\bm{u}_{h}\cdot\nabla)\bm{v}_{h}, \bm{u}_{h}- \bm{u}_{I} \right )-\left (\left ( \bm{u}_{h}-\bm{u}_{I}\right )\cdot\nabla\bm{v}_{h}, \bm{u} \right ) ] \\
& \quad -SR_{m}^{-1}\left ( \left ( \nabla\times \bm{B}\right )\times \left (\bm{B}_{h}-\bm{B}_{I} \right ), \bm{v}_{h}\right )-SR_{m}^{-1}\left (\bm{u}\times (\bm{B}_{h}-\bm{B}_{I}), \nabla_{h}\times \bm{C}_{h}\right ),
\end{align*}
and
\begin{align*}
H(\bm{u}_{h},& \bm{B}_{h}, \bm{u}, \bm{B}; \bm{u}-\bm{u}_{I}, \bm{B}-\bm{B}_{I}, p-p_{I}, r-r_{I}; \bm{v}_{h}, \bm{C}_{h})=\frac{1}{2} [ (  ( \bm{u}-\bm{u}_{I})\cdot \nabla  )\bm{u}_{h}, \bm{v}_{h} )\\&
+(  ( \bm{u}\cdot \nabla  ) (\bm{u}-\bm{u}_{I}), \bm{v}_{h})  -  ( (\bm{u}_{h}\cdot\nabla)\bm{v}_{h}, \bm{u}- \bm{u}_{I} )  - ( ( \bm{u}-\bm{u}_{I} )\cdot\nabla\bm{v}_{h}, \bm{u}  ) ]\\
&\quad+SR_{m}^{-1}\left ( (\nabla\times \bm{B})\times (\bm{B}_{I}-\bm{B}), \bm{v}_{h}\right )-SR_{m}^{-1} \left ( (\bm{u}-\bm{u}_{I})\times \bm{B}_{h},  \nabla_{h}\times \bm{C}_{h}\right )\\&\quad\quad
+SR_{m}^{-1}\left (\bm{u}\times (\bm{B}_{I}-\bm{B}), \nabla_{h}\times \bm{C}_{h} \right )+R_{e}^{-1}(\nabla(\bm{u}-\bm{u}_{I}), \nabla\bm{v}_{h})\\&\quad\quad\quad\quad\quad-(p-p_{I}, \nabla\cdot\bm{v}_{h})+(r-r_{I}, \nabla\cdot\bm{C}_{h}).
\end{align*}

Thanks to the energy law and the key estimate for the regularity of $\bm{B}_{h}$ (Theorem \ref{thm:L3}), norms $\|\bm{u}_{h}\|_{1}, \|\bm{B}_{h}\|_{d}, \|\bm{u}\|_{1}, \|\bm{B}\|_{0, 3}$ can be bounded by the source $\|\bm{f}\|_{-1}$.  Therefore $H$ and $G$ are bounded bilinear forms with coefficients which can be controlled  by $\|\bm{f}\|_{-1}$. 
Specifically, we have the boundedness
\begin{align*}
 |G(\bm{u}_{h}, \bm{B}_{h}, \bm{u},& \bm{B}; \bm{u}_{h}-\bm{u}_{I}, \bm{B}_{h}-\bm{B}_{I}; \bm{v}_{h}, \bm{C}_{h}) |\\&
\leq \Gamma_{1} \left ( \|\nabla(\bm{u}-\bm{u}_{I})\|^{2}+\|\nabla_{h}\times \left (\bm{B}_{h}-\bm{B}_{I}\right )\|^{2}\right )^{1/2}\left ( \|\nabla\bm{v}_{h}\|^{2}+\|\nabla_{h}\times \bm{C}_{h}\|^{2}\right )^{1/2},
\end{align*}
and
\begin{align}\label{eqn:H}
 |H(\bm{u}_{h}, \bm{B}_{h}, \bm{u}, &\bm{B}; \bm{u}-\bm{u}_{I}, \bm{B}-\bm{B}_{I}, p-p_{I}, r-r_{I}; \bm{v}_{h}, \bm{C}_{h}) |
\\\nonumber&  \leq \Gamma_{2}\left ( \|\bm{u}-\bm{u}_{I}\|_{1}^{2} +\|\bm{B}-\bm{B}_{I}\|^{2}+\|p-p_{I}\|^{2}+\|r-r_{I}\|^{2} \right )^{1/2} \|(\bm{v}_{h}, \bm{C}_{h})\|_{\tilde{\bm{X}}},
\end{align}
where
$$
\Gamma_{1}=C_{1}^{2}\left (\|\nabla\bm{u}_{h}\|+\|\nabla\bm{u}\| \right )+SR_{m}^{-1}C_1C_{2}\left (\|\nabla\times \bm{B}\|+\|\nabla\bm{u}\|\right ),
$$
and
\begin{align*}
\Gamma_{2}=C_{1}^{2}\left (\|\nabla\bm{u}_{h}\|+\|\nabla\bm{u}\| \right )&+SR_{m}^{-1}C_{1}C_{2}\|\nabla_{h}\times \bm{B}_{h}\|
\\&+SR_{m}^{-1}C_{1}\|\nabla\times \bm{B}\|_{0, 3}+SR_{m}^{-1}\|\bm{u}\|_{0, \infty}+2+R_{e}^{-1}.
\end{align*}
From the energy law, we have
$$
\|\nabla\bm{u}_{h}\|\leq R_{e}\|\bm{f}\|_{-1}, \quad \|\nabla\bm{u}\|\leq R_{e}\|\bm{f}\|_{-1},
$$
and
$$
\|\nabla\times \bm{B}\|\leq \sqrt{\frac{R_{e}R_{m}^{2}}{2S}}\|\bm{f}\|_{-1}, \quad \|\nabla_{h}\times \bm{B}_{h}\|\leq \sqrt{\frac{R_{e}R_{m}^{2}}{2S}}\|\bm{f}\|_{-1}.
$$
Therefore 
$$
\Gamma_{1}\leq \left ( 2C_{1}^{2}R_{e}+\sqrt{2}/2 C_{1}C_{2}\sqrt{R_{e}S}+SR_{m}^{-1}C_{1}C_{2}R_{e}\right )\|\bm{f}\|_{-1}.
$$

There are three remaining terms on the right hand side of \eqref{eqn:formal-system}, i.e. 
$$
I_{1}:=SR_{m}^{-1}\left ( (\nabla_{h}\times \bm{B}_{I}-\nabla\times \bm{B})\times \bm{B}_{h}, \bm{v}_{h}\right ),
$$
$$
I_{2}:=SR_{m}^{-2}\left (\nabla\times (\nabla\times \bm{B}- \nabla_{h}\times \bm{B}_{I}),  \bm{C}_{h} \right ),
$$
and
$$
I_{3}:=SR_{m}^{-1}(\nabla\times (\mathrm{id}-\mathbb{P})(\bm{u}\times \bm{B}), \bm{C}_{h}).
$$
Next, we estimate these three terms. The following lemma gives an estimate for the consistency term $\nabla\times \bm{B}-\nabla_{h}\times \bm{B}_{I}$. An analogous 2D version can be found in~\cite{Chen.L;Wang.M;Zhong.L.2014a}.
\begin{lemma}\label{consistency-curl}
We have the  estimate for the consistency of the  discrete adjoint operator 
$$
\|\nabla\times \bm{B}-\nabla_{h}\times \bm{B}_{I}\|\lesssim \|(\mathrm{id}-\mathbb{P})\nabla\times \bm{B}\|+h^{-1}\|\bm{B}-\bm{B}_{I}\|.
$$
\end{lemma}
\begin{proof}
We recall that $\mathbb{P}$ denotes the $L^{2}$ projection to $H^{h}_{0}(\mathrm{curl}, \Omega)$. We have
\begin{align*}
\|\nabla\times \bm{B}-\nabla_{h}\times \bm{B}_{I}\|&=\|\nabla\times \bm{B}-\mathbb{P}(\nabla\times \bm{B})+\mathbb{P}(\nabla\times \bm{B})-\nabla_{h}\times \bm{B}_{I}\|\\
&\leq \|(\mathrm{id}-\mathbb{P})\nabla\times \bm{B}\|+\|\mathbb{P}(\nabla\times \bm{B})-\nabla_{h}\times \bm{B}_{I}\|.
\end{align*}
For the second term, we use a dual estimate: for any $\bm{\phi}_{h}\in H^{h}_{0}(\mathrm{curl}, \Omega)$, 
\begin{align*}
(\mathbb{P}(\nabla\times \bm{B})-\nabla_{h}\times \bm{B}_{I}, \bm{\phi}_{h})&=(\nabla\times \bm{B}-\nabla_{h}\times \bm{B}_{I}, \bm{\phi}_{h})\\
&=(\bm{B}-\bm{B}_{I}, \nabla\times \bm{\phi}_{h})\\
&\leq \|\bm{B}-\bm{B}_{I}\|\|\nabla\times \bm{\phi}_{h}\|\\
&\lesssim h^{-1}\|\bm{B}-\bm{B}_{I}\|\| \bm{\phi}_{h}\|.
\end{align*}

This implies that $\|\mathbb{P}(\nabla\times \bm{B})-\nabla_{h}\times \bm{B}_{I}\|\lesssim h^{-1}\|\bm{B}-\bm{B}_{I}\|$ and the desired result follows.
\end{proof}

Lemma \ref{consistency-curl} implies the estimate for $I_{1}$:
$$
\left| I_{1}\right|\lesssim \left ( \|(\mathrm{id}-\mathbb{P})\nabla\times \bm{B}\|+ h^{-1}\|\bm{B}-\bm{B}_{I}\|\right )\|\bm{B}_{h}\|_{d}\|\bm{v}_{h}\|_{1}.
$$

We turn to the estimate for $I_{2}$: 
\begin{align*}
&\quad\left ( \nabla\times \nabla\times \bm{B}, \bm{C}_{h}\right )-(\nabla_{h}\times \bm{B}_{I}, \nabla_{h}\times \bm{C}_{h})
\\&=\left ( \nabla\times (\mathbb{P}+\mathrm{id}-\mathbb{P})\nabla\times \bm{B}, \bm{C}_{h}\right )-(\nabla_{h}\times \bm{B}_{I}, \nabla_{h}\times \bm{C}_{h})\\&
=\left ( \nabla\times \bm{B}-\nabla_{h}\times \bm{B}_{I}, \nabla_{h}\times \bm{C}_{h}\right )+(\nabla\times (\mathrm{id}-\mathbb{P})\nabla\times \bm{B}, \bm{C}_{h}).
\end{align*}
Using Lemma \ref{consistency-curl} again, we get
$$
\left |I_{2}\right |\lesssim  \left (  h^{-1}\|\bm{B}-\bm{B}_{I}\|+\|\left ( \mathrm{id}-\mathbb{P}\right )\nabla\times \bm{B}\|+\|\nabla\times(\mathrm{id}-\mathbb{P}) \nabla\times \bm{B}\|\right )\|\bm{C}_{h}\|_{d}.
$$

Moreover, we have a straightforward estimate for $I_{3}$:
$$
\left | I_{3} \right |\leq \|\nabla\times \left ( \mathrm{id}-\mathbb{P}\right )(\bm{u}\times \bm{B})\|\|\bm{C}_{h}\|.
$$

For any $\bm{B}\in H(\mathrm{div, \Omega})$, we define
$$
\|\bm{B}\|_{\mathrm{div}}^{2}:=\|\bm{B}\|^{2}+\|\nabla\cdot\bm{B}\|^{2}.
$$
\begin{lemma}\label{lem:rerm}
Assume  that $\|\bm{f}\|_{-1}$ is sufficiently small.
There exists $\mathcal{C}>0$ depending on $\Omega$,  $\|\bm{u}\|_{0, \infty}$ and $\|\bm{B}\|_{0, 3}$, such that for any $(\bm{u}_{I}, \bm{B}_{I})\in \tilde{\bm{X}}_{h}$, $(p_{I}, r_{I})\in \bm{Y}_{h}$, 
{\footnotesize\begin{align*}\footnotesize
\|&\bm{u}_{h}-\bm{u}_{I}\|_{1}^{2}+\|\bm{B}_{h}-\bm{B}_{I}\|_{d}^{2}+\|p_{h}-p_{I}\|^{2}+\|r_{h}-r_{I}\|^{2}\leq \mathcal{C} ( \|\bm{u}-\bm{u}_{I}\|_{1}^{2}
+\|\bm{B}-\bm{B}_{I}\|_{\mathrm{div}}^{2}+\|p-p_{I}\|^{2}\\&+\|r-r_{I}\|^{2}+ h^{-2}\|\bm{B}-\bm{B}_{I}\|^{2}+\| (\mathrm{id}-\mathbb{P})\nabla\times \bm{B}\|^{2}+\|\nabla\times (\mathrm{id}-\mathbb{P})\nabla\times \bm{B}\|^{2}+\|\nabla\times (\mathrm{id}-\mathbb{P})(\bm{u}\times \bm{B})\|^{2}  ).
\end{align*}}
\end{lemma}
\begin{proof}
Given $(\bm{u}, \bm{B}, p, r)$ and $(\bm{u}_{I}, \bm{B}_{I}, p_{I}, r_{I})$, the system \eqref{eqn:formal-system} can be seen as equations for $(\bm{u}_{h}-\bm{u}_{I}, \bm{B}_{h}-\bm{B}_{I}, p_{h}-p_{I}, r_{h}-r_{I})$. Compared with the nonlinear discrete system which we have analyzed, i.e. Problem \ref{prob:nonlinear-equi}, a new term $G$ appears on the left hand side and the fluid convection term has been absorbed into $G$.  

We assume that
\begin{align}\label{condition-f}
\|\bm{f}\|_{-1}\leq \mathrm{min}\left\{1/2R_{e}^{-1}, 1/2SR_{m}^{-2}\right \}\left (2C_{1}^{2}R_{e}+\sqrt{2}/2 C_{1}C_{2}\sqrt{R_{e}S}+SR_{m}^{-1}C_{1}C_{2}R_{e} \right )^{-1}.
\end{align}
A direct consequence  \eqref{condition-f} is $R_{e}\leq 1/2\Gamma_{1}^{-1}$ and $R_{m}\leq \left (1/2S\Gamma_{1}^{-1}\right )^{1/2}$.
Then
we have
{\footnotesize
$$
\left |G(\bm{u}_{h}, \bm{B}_{h}, \bm{u}, \bm{B}; \bm{v}_{h}, \bm{C}_{h}; \bm{v}_{h}, \bm{C}_{h})\right |\leq \frac{1}{2}R_{e}^{-1}\|\nabla \bm{v}_{h}\|^{2}+\frac{1}{2}SR_{m}^{-2}\|\nabla_{h}\times\bm{C}_{h}\|^{2},~ \forall \left (\bm{v}_{h}, \bm{C}_{h}\right )\in \tilde{\bm{X}}_{h},
$$}
then the left hand side {\footnotesize
\begin{align*}
\mathcal{A}&(\bm{w}_{h}, \bm{G}_{h}; \bm{v}_{h}, \bm{C}_{h}):=R_{e}^{-1}\left ( \nabla \bm{w}_{h}, \nabla\bm{v}_{h}\right )-SR_{m}^{-1}\left( \left ( \nabla_{h}\times \bm{G}_{h}\right )\times \bm{B}_{h}, \bm{v}_{h} \right )\\
&-SR_{m}^{-1} \left ( \bm{w}_{h}\times \bm{B}_{h}, \nabla_{h}\times \bm{C}_{h}\right )+SR_{m}^{-2}\left (\nabla_{h}\times \bm{G}_{h}, \nabla_{h}\times \bm{C}_{h} \right )+G(\bm{u}_{h}, \bm{B}_{h}, \bm{u}, \bm{B}; \bm{w}_{h}, \bm{G}_{h}; \bm{v}_{h}, \bm{C}_{h})
\end{align*}}
defines a bounded coercive bilinear form for fixed $\bm{u}_{h}$, $\bm{B}_{h}$, $\bm{u}$ and $\bm{B}$. The boundedness constant depends on $\|\bm{u}_{h}\|_{1}, \|\bm{u}\|_{1}$, $\|\nabla_{h}\times \bm{B}_{h}\|$ and $\|\nabla\times \bm{B}\|_{0, 3}$, which further depend on $\|\bm{f}\|_{-1}$.

For the right hand sides, $H(\bm{u}_{h}, \bm{B}_{h}, \bm{u}, \bm{B}; \bm{u}-\bm{u}_{I}, \bm{B}-\bm{B}_{I}, p-p_{I}, r-r_{I}; \cdot)$ can be regarded as a bounded linear functional on $\tilde{\bm{X}}_{h}$ for fixed $\bm{u}_{h}, \bm{B}_{h}, \bm{u}, \bm{B}, \bm{u}_{I}, \bm{B}_{I}$, and the dual norm  can be bounded by 
$$
\Gamma_{2}\left ( \|\bm{u}-\bm{u}_{I}\|_{1}^{2}+\|\bm{B}-\bm{B}_{I}\|_{}^{2}+\|p-p_{I}\|^{2}+\|r-r_{I}\|^{2} \right )^{1/2},
$$
due to \eqref{eqn:H}.
Moreover, given $\bm{u}_{h}-\bm{u}_{I}$ and $\bm{B}_{h}-\bm{B}_{I}$,   $\left  (\nabla\cdot (\bm{u}_{h}-\bm{u}_{I}), q_{h}\right )$ and $\left  (\nabla\cdot (\bm{B}_{h}-\bm{B}_{I}), s_{h}\right )$ are bounded linear functionals on $Q_{h}$ and $L^{2}_{h}$ respectively, with dual norms $\|\nabla\cdot (\bm{u}_{h}-\bm{u}_{I})\|$ and $\|\nabla\cdot (\bm{B}_{h}-\bm{B}_{I})\|$. From the estimates for $I_{1}$, $I_{2}$ and $I_{3}$, dual norms of these three terms can be bounded by
{\footnotesize
$$
\max\left \{  \|\nabla\times (\mathrm{id}-\mathbb{P})\nabla\times \bm{B}\|+  h^{-1}\|\bm{B}-\bm{B}_{I}\|+    \| (\mathrm{id}-\mathbb{P})\nabla\times \bm{B}\|, \|\nabla\times (\mathrm{id}-\mathbb{P})(\bm{u}\times \bm{B})\| \right \},
$$}
up to a positive constant.


From a general argument of the Brezzi theory, we see that the norms of the solution of \eqref{eqn:formal-system}, i.e.,
$$
\|\left (\bm{u}_{h}-\bm{u}_{I}, \bm{B}_{h}-\bm{B}_{I}\right )\|_{\tilde{\bm{X}}}^{2}+ \|\left (p_{h}-p_{I}, r_{h}-r_{I}\right )\|_{ \bm{Y}}^{2}
$$
 can be bounded by the dual norm   of the right hand side. This completes the proof.
\end{proof}

Combining triangular inequalities and the estimate
$$
\|\nabla\times \bm{B}-\nabla_{h}\times \bm{B}_{h}\|\leq \|\nabla\times \bm{B}-\nabla_{h}\times \bm{B}_{I}\|+\|\nabla_{h}\times \left ( \bm{B}_{I}-  \bm{B}_{h}\right )\|,
$$
we obtain the following quasi-optimal estimate.
\begin{theorem}\label{thm:quasi-orthogonality}
Assume that the condition \eqref{condition-f} holds.  There exists a generic positive constant $\mathcal{C}>0$ depending on $\Omega$, $\|\bm{f}\|_{-1}$, $\|\bm{u}\|_{0, \infty}$ and $\|\bm{B}\|_{0, 3}$, such that for any $(\bm{u}_{I}, \bm{B}_{I})\in \tilde{\bm{X}}_{h}$, $(p_{I}, r_{I})\in \bm{Y}_{h}$, 
\begin{align}\footnotesize\nonumber
\|\bm{u}-&\bm{u}_{h}\|_{1}^{2}+\|\bm{B}-\bm{B}_{h}\|_{\mathrm{div}}^{2}+\|\nabla\times \bm{B}-\nabla_{h}\times \bm{B}_{h}\|^{2}+\|p-p_{I}\|^{2}+\|r-r_{I}\|^{2}\\\nonumber
&\leq \mathcal{C} ( \|\bm{u}-\bm{u}_{I}\|_{1}^{2}
+\|\bm{B}-\bm{B}_{I}\|_{\mathrm{div}}^{2}+\|p-p_{I}\|^{2}+\|r-r_{I}\|^{2}+ h^{-2}\|\bm{B}-\bm{B}_{I}\|^{2}\\&\label{final-estimate-i}
+ \| (\mathrm{id}-\mathbb{P})\nabla\times \bm{B}\|^{2}+\|\nabla\times (\mathrm{id}-\mathbb{P})\nabla\times \bm{B}\|^{2}+\|\nabla\times (\mathrm{id}-\mathbb{P})(\bm{u}\times \bm{B})\|^{2}  ).
\end{align}
\end{theorem}
We remark that $\left \|\nabla\times \bm{B}-\nabla_{h}\times \bm{B}_{h}\right \|=\left \|R_{m}\left ( \bm{j}-\bm{j}_{h}\right )\right \|$ yields an $L^{2}$ error estimate for the current density $\bm{j}$.


The last step is to estimate the convergence order based on the polynomial approximation theory. We recall the following approximation result.
\begin{lemma}\label{lem:approximation-p}
Assume that $H^{h}(\mathrm{curl}, \Omega)$ contains piecewise polynomials of degree $s$. Then the $L^{2}$ projection $\mathbb{P}$ satisfies the approximation property
$$
\left \|\bm{\phi}-\mathbb{P}\bm{\phi}\right \|+ h\left \|\nabla\times \left ( \bm{\phi}-\mathbb{P}\bm{\phi}\right )\right \|\lesssim h^{s+1}\|\bm{\phi}\|_{s+1}, \quad \forall \bm{\phi}\in {H}^{s+1}(\Omega)^{3}.
$$
\end{lemma}
The proof is almost the same as the classical result of $L^{2}$ projections for  Lagrange elements. For completeness, we include the proof here.
\begin{proof}
Let $\Pi_{\mathrm{curl}}^{h}$ be a bounded interpolation operator to $H^{h}(\mathrm{curl}, \Omega)$, for example, defined in \cite{falk2014local}. Then we have
$$
\left \|\nabla\times\left ( \bm{\phi}-\mathbb{P}\bm{\phi}\right )\right \|\leq\left \|\nabla\times\left ( \bm{\phi}-\Pi_{\mathrm{curl}}^{h}\bm{\phi}\right )\right \|+\left \|\nabla\times \Pi_{\mathrm{curl}}^{h}\left (\bm{\phi}-\mathbb{P}\bm{\phi}\right )\right \|.
$$
For the first term on the right hand side, 
$$
\left \|\nabla\times\left ( \bm{\phi}-\Pi_{\mathrm{curl}}^{h}\bm{\phi}\right )\right \|\lesssim h^{s}\|\bm{\phi}\|_{s+1}.
$$
For the second, we use the inverse estimate to get
$$
\left \|\nabla\times \Pi_{\mathrm{curl}}^{h}\left (\bm{\phi}-\mathbb{P}\bm{\phi}\right )\right \|\lesssim h^{-1}\left \| \Pi_{\mathrm{curl}}^{h}\left (\bm{\phi}-\mathbb{P}\bm{\phi}\right )\right \|\lesssim h^{s}\|\bm{\phi}\|_{s+1}.
$$
This implies $h\left \|\nabla\times\left ( \bm{\phi}-\mathbb{P}\bm{\phi}\right )\right \|\lesssim h^{s+1}\|\bm{\phi}|_{s+1}$.

On the other hand, the approximation
$$
\left \|\bm{\phi}-\mathbb{P}\bm{\phi}\right \|\lesssim h^{s+1}\|\bm{\phi}\|_{s+1}
$$
follows directly from the property of the $L^{2}$ projection operator. This completes the proof.
\end{proof}

In the following discussions, we assume that $H^{h}(\mathrm{curl}, \Omega)$, $H^{h}(\mathrm{div}, \Omega)$ and $L^{2}_{h}(\Omega)$ contain piecewise polynomials of degree $r_{1}$, $r_{2}$ and $r_{3}$ respectively.  From the construction of discrete de Rham complexes,  we have $r_{i}=r_{i+1}$ or $r_{i}=r_{i+1}+1$ where $i=1, 2$.  We assume that the approximation space $\bm{V}_{h}$ for the velocity contains piecewise polynomials of degree $s_{u}$ and the discrete pressure space $Q_{h}$ contains piecewise polynomials of degree $s_{p}$. 

We estimate the projection error on the right hand side of \eqref{final-estimate-i} based on Lemma \ref{lem:approximation-p}:
$$
\|\nabla\times (\mathrm{id}-\mathbb{P})(\bm{u}\times \bm{B})\|\lesssim h^{r_{1}}\|\bm{u}\times \bm{B}\|_{r_{1}+1},
$$
$$
\|\nabla\times (\mathrm{id}-\mathbb{P})\nabla\times \bm{B}\|\lesssim h^{r_{1}}\|\nabla\times \bm{B}\|_{r_{1}+1},
$$

Consequently, we have
\begin{align}\nonumber
\|\bm{u}-\bm{u}_{h}\|_{1}^{2}+&\|\bm{B}-\bm{B}_{h}\|_{\mathrm{div}}^{2}+\|\nabla\times \bm{B}-\nabla_{h}\times \bm{B}_{h}\|^{2}+\|p-p_{I}\|^{2}+\|r-r_{I}\|^{2}\\&\nonumber
\leq \mathcal{C} ( h^{2s_{u}}\|\bm{u}\|_{s_{u}+1}+h^{2s_{p}+2}\|p\|_{s_{p}+1}^{2}+h^{2r_{2}}\|\bm{B}\|_{r_{2}+1}^{2}+h^{2r_{1}}(\|\bm{u}\times \bm{B}\|_{r_{1}+1}^{2}\\&\label{error} \quad\quad\quad+\|\nabla\times \bm{B}\|_{r_{1}+1}^{2})+h^{2r_{3}+2}\|r\|_{r_{3}+1}^{2}).
\end{align}

Based on the error estimate \eqref{error}, we can get balanced errors by choosing finite elements such that $r_{1}=r_{2}=r_{3}+1=s_{u}=s_{p}+1$. One particular choice is to use  BDM spaces for the magnetic field $\bm{B}$,  N\'{e}d\'{e}lec spaces of the first kind for the electric field $\bm{E}$. The pressure multiplier $p$ and the magnetic multiplier $r$ may be chosen to have the same order. 

The above analysis  excludes  the lowest order Raviart-Thomas element, but includes the case of the lowest order BDM element. 
We believe that this  restriction is only technical but a more refined estimate is beyond the scope of this paper.


\section{Concluding remarks}\label{sec:conclusion}
In this paper we considered the mixed finite element discretizations of
the stationary MHD system.  Compared to the time-dependent system, the
Gauss's law of magnetic field is an independent equation which cannot
be derived from the Faraday's law. Therefore classical techniques of
Lagrange multipliers are employed to impose the Gauss's law.
The structure-preserving discretization proposed in this paper for the
stationary MHD system preserves both the discrete energy law and most
importantly the Gauss's law $\nabla\cdot \bm{B}=0$.


We note that we can also use a formulation based on $\bm{B}$ and
$\bm{E}$, which is similar to the time-dependent case studied in
\cite{hu2014stable}.  But the well-posedness of such a formulation can
only be established when the Reynolds number $R_{e}$ is assumed to be
sufficiently small. To remove such an undesirable constraint, we
proposed the new formulation using $\bm{B}$ and $\bm{j}$ as the
variables.  Such a formulation was partially motivated by the fact
that the energy is given in terms of $\|\bm j\|$ rather than $\|\bm
E\|$.

These two formulations look similar.  In the finite element
discretization of both cases, we have $\bm{j}=\bm{E}+\mathbb{P}(\bm{u}\times \bm{B})$
(only one variable of $\bm{E}$ and $\bm{j}$ is explicitly used in one
scheme). This is an equation in $H_{0}^{h}(\mathrm{curl},
\Omega)$. The current density $\bm{j}$ and the electric field $\bm{E}$
differ by a nonlinear term, which is projected to
$H_{0}^{h}(\mathrm{curl}, \Omega)$.  But the resulting
 formulations are different due
to the different treatments of the nonlinear term
$\mathbb{P}(\bm{u}\times \bm{B})$ in the discretization
of the Lorentz force term.  We note that in the 
formulation proposed in \cite{hu2014stable}, the Lorentz force term $(\bm{j}, \bm{v}\times \bm{B})$ is
discretized as
$$
(\bm{E}+\bm{u}\times \bm{B}, \bm{v}\times \bm{B}).
$$
Whereas in the formulation proposed in this paper, the corresponding
discretization is as
$$
(\bm{E}+\mathbb{P}(\bm{u}\times \bm{B}), \mathbb{P}(\bm{v}\times \bm{B})).
$$
It is easy to see that these two discretizations are indeed different. 

Similar differences can be also found at other places.  A key point to
get well-posedness is the cancellation of the symmetric nonlinear
coupling terms. Under such a restriction, other parts of the schemes
also have to be different according to the different Lorentz force
terms.  Indeed the energy estimates of these two kinds of formulations
have already shown the difference.  The energy estimates of the
formulation in \cite{hu2014stable} involve $\|\bm{E}+\bm{u}\times
\bm{B}\|^{2}$, while the formulation in this paper involves
$\|\bm{j}\|^{2}=\|\bm{E}+\mathbb{P}(\bm{u}\times \bm{B})\|^{2}$.

As a result of these differences, a careful analysis indicates that
the well-posedness of the formulation proposed in this paper can be
established without any assumption on the size of $R_{e}$.


\section*{Acknowledgement}
The authors would like to thank Mr. Juncai He, Prof. Ragnar Winther and
Dr.~Shuonan Wu for helpful discussions, and the anonymous referees for
valuable suggestions, which have greatly improved the quality of the
paper.

\bibliographystyle{amsplain}     
\bibliography{MHD-copy}{}   

\end{document}